\documentclass[12pt,a4paper]{amsart}

\RequirePackage[ansinew]{inputenc}
\usepackage[english]{babel}
\usepackage{amsmath,amsfonts,amsthm,amssymb}
\usepackage[top=3cm,left=2.5cm,right=3cm,bottom=3cm]{geometry}
\usepackage[usenames,dvipsnames]{color}
\usepackage[colorlinks=true,linkcolor=blue,citecolor=blue]{hyperref}

\usepackage{tikz} 

\newcommand{\C}{\mathbb{C}}
\newcommand{\R}{\mathbb{R}}

\newcommand{\N}{\mathbb{N}}

\newcommand{\del}{\partial}

\newcommand{\ddc}{\text{dd}^c}

\newcommand{\supp}{\text{\normalfont supp }}

\newcommand{\Leb}{\text{\normalfont Leb}}

\newcommand{\pr}{\mathbb{P}}

\newcommand{\re}{\text{\normalfont Re }}
\newcommand{\im}{\text{\normalfont Im }}

\newtheorem{theorem}{Theorem}[section]
\newtheorem{proposition}[theorem]{Proposition}
\newtheorem{corollary}[theorem]{Corollary}
\newtheorem{lemma}[theorem]{Lemma}

\newtheorem*{theorem*}{Theorem}

\theoremstyle{definition}
\newtheorem{definition}[theorem]{Definition}

\newtheorem{remark}[theorem]{Remark}

\numberwithin{equation}{section}

\title{Commuting pairs of endomorphisms of $\pr^2$}

\author{Lucas Kaufmann}
\address{Sorbonne Universités, UPMC Univ Paris 06, IMJ-PRG, UMR 7586 CNRS, Univ Paris Diderot, Sorbonne Paris Cité, F-75005, Paris, France}
\email{lucas.kaufmann@imj-prg.fr}
\date{}

\begin{document}
\maketitle

\begin{abstract}
We consider commuting pairs of holomorphic endomorphisms of $\pr^2$ with disjoint sequence of iterates. The remaining case to be studied is when their degrees coincide after some number of iterations. We show in this case that they are either commuting Lattès maps or commuting homogeneous polynomial maps of $\C^2$ inducing a Lattès map on the hyperplane at infinity.
\end{abstract}

\section{Introduction and main result}
The study of functional equations is at the origin of the early developments of the iteration theory of polynomials and rational functions, carried out by Fatou, Julia, Ritt and others. Among these equations, the commutation relation
\begin{equation} \label{eq:commutation}
f \circ g = g \circ f
\end{equation}
is one of the simplest and most commonly encountered in mathematics and its applications. It can also be thought as a sort ``integrability condition" for the dynamical system defined by $f$ (see \cite{veselov:integrable}). 

The case when $f$ and $g$ are rational functions on the Riemann sphere already appears in the work of the above mentioned authors (see \cite{fatou}, \cite{julia}, \cite{ritt}) and it was revisited with the use of a modern language by Eremenko in \cite{eremenko:functional-equations}.

Motivated by a question of S. Smale, asking whether the centralizer of a ``typical'' diffeomorphism of a compact manifold is trivial,  Dinh and Sibony studied in  \cite{dinh-sibony:endo-permutables} pairs of endomorphisms $f,g:\pr^k \to \pr^k$ of the complex projective space $\pr^k$ of degree $d_f,d_g \geq 2$ satisfying (\ref{eq:commutation}) and the supplementary condition
\begin{equation} \label{eq:disjoint-degrees}
d_f^n \neq d_g^m, \text{ for all } n,m \geq 1.
\end{equation}

Their main result says that these maps come from affine maps of $\C^k$ after taking its quotient by a discrete group of affine transformations.  As a corollary one gets that such endomorphisms are critically finite, that is, their critical set is pre-periodic. This remark allows us to show that condition (\ref{eq:disjoint-degrees}) is necessary for their conclusion if the dimension $k$ is $3$ or higher, by exhibiting a nontrivial commuting pair of  endomorphisms of same degree on $\pr^3$  that are not critically finite, see  \cite{dinh-sibony:endo-permutables}.

The aim of this paper is to describe commuting holomorphic endomorphisms of $\pr^2$ assuming a weaker condition, namely that they do not share an iterate, that is,
\begin{equation} \label{eq:disjoint-iterates} \tag{$\star$}
f^n \neq g^m, \text{ for all } n,m \geq 1.
\end{equation}

We will establish in this setting the following

\begin{theorem} \label{thm:main-thm}
Let $f$ and $g$ be two commuting holomorphic endomorphisms of $\pr^2$ of degree $d_f, d_g \geq 2$. Assume that $f^n \neq g^m$ for all $ n,m \geq 1$ and that $d_f^k = d_g^\ell$ for some $ k,\ell \geq 1$. 
Then both $f$ and $g$ are of one of the following types:
\begin{enumerate}
\item Commuting Lattès maps of $\pr^2$.
\item Lifts of commuting Lattès maps of $\pr^1$:  $f$ and $g$ are homogeneous polynomial endomorphisms of $\C^2$ that extend holomorphically to $\pr^2$ and such that the induced maps on the hyperplane at infinity are Lattès maps of $\pr^1$.
\end{enumerate}
\end{theorem}

Recall that a holomorphic endomorphism $f: \pr^k \to \pr^k$ of degree $d$ is called a {\it Lattès map} (cf. \cite{berteloot-loeb:lattes}) if there is a complex crystallographic group $\mathcal A$, an  affine map $\Lambda_f$ with linear part $\sqrt d \, U$ with $U$ unitary and a ramified covering $\varphi: \C^k \to \pr^k$ such that the diagram
\vspace*{-20pt}
\begin{center}
 \begin{tikzpicture}[node distance=2cm, thick, auto]
  \node (C1) {$\C^k$};
  \node (C2) [right of=C1] {$\C^k$}; 
  \node (U1) [below of=C1] {$\pr^k$};
  \node (U2) [below of=C2] {$\pr^k$};
  \draw[->] (C1) to node {$\Lambda_f$} (C2);
  \draw[->] (U1) to node {$f$} (U2);
  \draw[->] (C1) to node [swap] {$\varphi$} (U1);
  \draw[->] (C2) to node {$\varphi$} (U2);
\end{tikzpicture}
\end{center}
commutes and such that $\mathcal A$ acts transitively on the fibers of $\varphi$.

Notice that, in particular, we have $\C^k / \mathcal A \simeq \pr^k$. In the two dimensional setting, the groups $\mathcal A$ such that  $\C^2 / \mathcal A \simeq \pr^2$  were  classified in \cite{tokunaga-yoshida}. Using this classification, it is possible to give a precise description of Lattès maps of $\pr^2$ (see \cite{dupont:lattes} and \cite{rong:lattes}). For a complete study of Lattès maps of $\pr^1$ the reader may consult \cite{milnor:lattes}.

The strategy of the proof, which is inspired by the methods in \cite{dinh-sibony:endo-permutables} and \cite{eremenko:functional-equations}, relies on the study of the lamination of the Julia set $J_2$ and the Green measure $\mu$ of the pair $f,g$. We may suppose at first that $d_f = d_g$. We start by considering the Poincaré map $\varphi: \C^2 \to \pr^2$ of $f$ and $g$ associated with a common repelling periodic point. This map gives a semiconjugation between $f$, $g$ and triangular automorphisms $\Lambda_f, \Lambda_g$ of $\C^2$. We consider then the closed abelian Lie group of triangular automorphisms of $\C^2$ generated by $\Lambda_f \circ \Lambda_g^{-1}$, which preserves several dynamical objects. Condition  (\ref{eq:disjoint-iterates}) insures that this group is positive dimensional whereas the condition that $d_f= d_g$ implies that it is compact. This procedure can be applied to different periodic points, giving to $J_2$ and $\mu$ a laminar structure induced by the orbits of these groups.

The next step consists of a detailed analysis of the geometry of the lamination defined above. We begin by showing that the dimension of its leaves is $m=3$ or $4$, see Proposition \ref{prop:lamination}. When $m=4$ the measure $\mu$ is smooth in some open set and we conclude, using a theorem of Berteloot and Dupont, that both $f$ and $g$ are Lattès maps of $\pr^2$. The case $m=3$ is more delicate and Section \ref{sec:m=3} is devoted to it. We prove in this setting that $f$ and $g$ possess a totally invariant line. This is done by analyzing separately the cases when the leaves are strictly pseudoconvex and when they are Levi-flat. Once this step is achieved, we reduce the problem to the case where $f$ and $g$ are polynomial maps of $\C^2$, see Theorem \ref{thm:3-laminated-polynomial}. We then apply a theorem of Dinh saying that a polynomial map of $\C^k$ that extends holomorphically to $\pr^k$ and whose maximal order Julia set contains a piece of hypersurface is a lift of a Lattès maps of $\pr^{k-1}$.

\textbf{Acknowledgements.} This work was supported by a grant from Région Île-de-France.

\section{Holomorphic endomorphisms of $\pr^k$: general facts}
We recall here the basic dynamical objects associated with a holomorphic endomorphism of $\pr^k$. For further details the reader may consult \cite{sibony:dynamique-Pk} and \cite{dinh-sibony:cime}.

We start by considering a polynomial endomorphism $f:\C^k \to \C^k$ of algebraic degree $d \geq 2$ which can be extended holomorphically to $\pr^k$. We define its {\it Green function} by $$G(z) = \lim _{n \to \infty} \frac{1}{d^n} \log^+ \|f^n(z)\|.$$
It is a Hölder continuous plurisubharmonic function on $\C^k$ vanishing exactly at the points of bounded orbit. Moreover we can show that  $K = \{G=0\}$ is a compact subset of $\C^k$. The {\it Green current of} $f$ is the positive closed (1,1)-current  defined by $T = dd^c G$.

For a holomorphic endomorphism $f:\pr^k \to \pr^k$ of degree $d \geq 2$ its {\it Green current} is defined as the limit $$ T = \lim _{n \to \infty} \frac{1}{d^n} (f^n)^*\omega,$$
where $\omega$ is the Fubini-Study form. We get the same limiting current if we replace $\omega$ in the above formula by any positive closed $(1,1)$-current of mass $1$ and of bounded local potential. As in the polynomial case we can show that the local potentials of the Green current are Hölder continuous.

In both polynomial and holomorphic cases we get higher order Green currents by taking wedge products of $T$ with itself: $T^\ell = T \wedge \cdots \wedge T$ ($\ell$ times) for $1\leq \ell \leq k$. They are well-defined since $T$ admits continuous local potentials. When $\ell = k$ we call $\mu = T^k$, the {\it Green measure} of $f$. It is the unique invariant probability measure of maximal entropy for the dynamical system defined by $f$ \cite{briend-duval:mesure}.

The {\it Julia set of order $s$} of $f$, denoted by $J_s$, is by definition the support of the current $T^s$, $1\leq \ell \leq k$. In this work we will focus mainly on the maximal order Julia set $J_k$.

The following result will be useful to us. Its proof is inspired by \cite{dinh-sibony:equidistribution-ENS}.

\begin{proposition} \label{prop:currentJk}
Let $S$ be a positive closed $(1,1)$-current on $\pr^k$. Assume that there are $n_i \nearrow +\infty$ and positive closed $(1,1)$-currents $S_{n_i}$ of unit mass supported by $J_k$ such that $S = \lim_{i \to \infty} d^{-n_i} (f^{n_i})^* S_{n_i}$. Then $S = T$. In particular, if $S$ is a positive closed $(1,1)$-current of unit mass supported by $J_k$ and $f^*S = d \cdot S$, then $S = T$.
\end{proposition}

\begin{proof}
Recall that a function $u: \pr^k \to \R \cup \{-\infty\}$ is called p.s.h.\ modulo $T$ if $u$ is locally a difference of a p.s.h.\ function and a potential of $T$. Such a function is upper semi-continuous, it belongs to $L^1(\pr^k)$ and $T + \ddc u$ is a positive current. Modulo $T$ p.s.h.\ functions satisfy the inequalities
\begin{equation*}
\max_{\pr^k} u \leq c(1 + \|u\|_{L^1(\pr^k)}) \text{ and } \int u\, d \mu \leq c(1 + \|u\|_{L^1(\pr^k)}) 
\end{equation*}
where $c$ is a dimensional constant (see \cite{dinh-sibony:equidistribution-ENS}) and $\mu$ is the Green measure.


By the $\ddc$-Lemma there are functions $u_i$ satisfying $T = S_{n_i} - \ddc u_i$ and $\max_{\pr^k} u_i = 0$. The family $\{u_i\}$ is bounded in $L^1(\pr^k)$ and  $d^{-n_i} u_i \circ f^{n_i}$ converges  in $L^1(\pr^k)$ to a modulo $T$ p.s.h. function $u$ satisfying $T = S - \ddc u$. 

Since $u_i \leq 0$ we have $u \leq 0$. From the second inequality above, the sequence $u_i$ is bounded in $L^1(\mu)$. By the invariance of $\mu$ by $f$ we get $\|d^{-n_i} u_i \circ f^{n_i}\|_{L^1(\mu)} = d^{-n_i}  \|u_i \|_{L^1(\mu)} \to 0$, so by Fatou's Lemma $\int u \; d \mu \geq 0$ . Together with the fact that $u \leq 0$ we get $\int u \; d \mu = 0$ and by upper semicontinuity  $u = 0$ on $\supp \mu = J_k$. In particular $u$ is finite over $J_k$.

On the other hand, since $\supp S \subset J_k$ we have $T|_{\pr^k \setminus J_k} = (- \ddc u)|_{\pr^k \setminus J_k}$ showing that $u$ is locally bounded outside $J_k$. We conclude that $u$ is finite on $\pr^k$. In particular the current $S$ has zero Lelong number everywhere on $\pr^k$. 

From the upper semicontinuity of the Lelong number we get $$\lim _{i \to \infty} \sup_{p \in \pr^k} \, \nu(d^{-n_i}(f^{n_i})^* S_{n_i},p) \leq \sup_{p \in \pr^k} \nu(S,p) = 0.$$ An equidistribuition theorem (see \cite{guedj:equidistribution}), \cite{dinh-sibony:equidistribution-ENS}, \cite{taflin:equidistribution}) gives $ \lim_{i \to \infty} d^{-n_i} (f^{n_i})^* S_{n_i} = T$, so $S=T$.
\end{proof}

\begin{remark}
The existence of a totally invariant $(1,1)$-current supported by $J_k$ as in the above proposition is quite special. Indeed, in this case we have that $J_1 = \supp T \subset J_k$, so $J_k = J_1$ and the Julia sets of all orders coincide. Some examples of this phenomenon appear in \cite{fornaess-sibony:examples}.
\end{remark}

The next lemma is well known and it is useful in producing totally invariant currents.

\begin{lemma} \label{lemma:cesaro}
Let $R$ be a positive closed $(1,1)$-current of mass one on $\pr^k$. Consider the sequence of Cesaro sums $S_n = \frac{1}{n} \sum_{j=0}^{n-1} d^{-j} (f^j)^* R$. Then $(S_n)_{n \geq 1}$ is relatively compact and every cluster value $S$ is totally invariant, that is, $f^*S = d \cdot S$.
\end{lemma}

\begin{proof}
The first statement follows directly from the fact that  each $S_n$ is of mass one. Let $S$ be a cluster value and take $n_\ell \nearrow + \infty$ such that $S_{n_\ell} \to S$ as $\ell \to \infty$. Notice that $n S_n = \sum_{j=0}^{n-1} d^{-j} (f^j)^* R $ for every $n$, so that $$d^{-1} f^*S_{n_\ell} = \frac{1}{n_\ell} \sum_{j=0}^{n_\ell-1} d^{-{j+1}} (f^{j+1})^* R =  \frac{1}{n_\ell} \sum_{j=1}^{n_\ell} d^{-j} (f^j)^* R = \frac{n_\ell + 1}{n_\ell} S_{n_\ell + 1} - \frac{1}{n_\ell} R.$$

Since the term on the left-hand side converges to $d^{-1}f^* S$ and the last term on the right-hand side converges to $0$ we need to show that $\frac{n_\ell + 1}{n_\ell} S_{n_\ell + 1} \to S$ as $\ell \to \infty$. Writing $$\frac{n_\ell + 1}{n_\ell} S_{n_\ell + 1} = \frac{1}{n_\ell}[(n_\ell+1)S_{n_\ell +1} - n_\ell S_{n_\ell}] + S_{n_\ell}$$ this amounts to showing that $\frac{1}{n_\ell}[(n_\ell+1)S_{n_\ell +1} - n_\ell S_{n_\ell}]  \to 0$.

We have that $$\frac{1}{n_\ell}[(n_\ell+1)S_{n_\ell +1} - n_\ell S_{n_\ell}]  = \frac{1}{n_\ell}[   \sum_{j=0}^{n_\ell} d^{-j} (f^j)^* R -   \sum_{j=0}^{n_\ell - 1} d^{-j} (f^j)^* R] = \frac{1}{n_\ell} d^{-n_\ell} (f^{n_\ell})^*R \to 0,$$ since the currents $d^{-n_\ell} (f^{n_\ell})^*R$ have all mass one. This concludes the proof.
\end{proof}

The following theorem shows that the Green measure describes the distribution of pre-images of points outside an exceptional set, giving also a precise description of this set, see \cite{dinh-sibony:cime}.

\begin{theorem} [Equidistribution of pre-images] \label{thm:equidistribution-preimages}
Let $f:\pr^k \to \pr^k$ be a holomorphic endomorphism of degree $d\geq 2$ and $\mu$ its Green measure.  Then there is a proper algebraic subset $\mathcal E$ of $\pr^k$ such that the measures $$\mu_n^a  = \frac{1}{d^{nk}} (f^n)^* \delta_a = \frac{1}{d^{nk}} \sum_{x \in f^{-n}(a)} \delta_x $$ converge to $\mu$ as $n \to \infty$ if and only if $a \notin \mathcal E$.

Moreover, the set  $\mathcal E$ is totally invariant $f^{-1}(\mathcal E) = f(\mathcal E) = \mathcal E$ and is maximal in the sense that if $E$ is a proper algebraic subset of $\pr^n$ such that $f^{-n}(E) \subset E$ for some $n \geq 1$, then $E \subset \mathcal E$.
\end{theorem}

It is known (see  \cite{fornaess-sibony:dynamics1}, \cite{cerveau-lins-neto:hyp-exc}) that the codimension one components of $\mathcal E$ form a hypersurface of degree at most $k+1$ and it is conjectured that they are all linear. In dimension $2$ the situation is completely understood.

\begin{theorem} [Fornæss-Sibony \cite{fornaess-sibony:dynamics1}, Cerveau-Lins Neto \cite{cerveau-lins-neto:hyp-exc}] \label{thm:exceptional-set}
Let $f:\pr^2 \to \pr^2$ be a holomorphic endomorphism of degree $d\geq 2$ and let $\mathcal E$ be its exceptional set. Then the one dimensional part of $\mathcal E$ is a union of at most three lines. In particular, if no iterate of $f$ is conjugated to a  polynomial map then $\mathcal E$ is finite.
\end{theorem}


We state a last general result for further use.

\begin{proposition} \cite{dinh:lattes} \label{prop:complex-curves-J}
Let $f$ be a polynomial endomorphism of $\C^k$ of degree $d \geq 2$ that extends to a holomorphic endomorphism of $\pr^k$. Then there is no complex manifold contained in the maximal order Julia set $J_k$ which passes through a repelling periodic point.
\end{proposition}

\begin{remark}
For non-polynomial endomorphisms, the above proposition is no longer true, as the Lattès maps show. Other non-Lattès examples can be found in \cite{fornaess-sibony:examples}. Proposition \ref{prop:J1=J2} gives a related result for holomorphic endomorphisms of $\pr^2$.
\end{remark}

\subsection*{Commuting pairs}
For a commuting pair of endomorphisms, the several dynamical objects introduced above coincide. This is made precise by the two propositions below.

Given a commuting pair $f,g: \pr^k \to \pr^k$ it is easy to see that we can always find commuting polynomial lifts $F,G: \C^{k+1} \to \C^{k+1}$, that is, $F$ and $G$ are commuting homogeneous polynomial mappings, $\pi \circ F = \pi \circ f$ and $\pi \circ G = \pi \circ g$, where $\pi : \C^{k+1} \setminus \{0\} \to \pr^k$ is the canonical projection.

\begin{proposition} \cite{dinh-sibony:endo-permutables} \label{prop:same-G}
Let $f,g: \pr^k \to \pr^k$ be a commuting pair of endomorphisms of degree $d_f,d_g \geq 2$. Choose lifts $F,G: \C^{k+1} \to \C^{k+1}$ such that $F \circ G = G \circ F$ and denote by $G_F$ and $G_G$  their corresponding Green functions. Then $G_F = G_G$. In particular $T_f = T_g$ and $J_s(f) = J_s(g)$ for every $s=1,\ldots ,k.$
\end{proposition}

\begin{proposition} \label{prop:same-E}
 Let $f,g: \pr^k \to \pr^k$ be a commuting pair of endomorphisms. Denote by $\mathcal E_f$ and $\mathcal E_g$ the corresponding exceptional sets. Then $\mathcal E_f = \mathcal E_g$.
\end{proposition}

\begin{proof}
We will show that $\mathcal E_f \subset \mathcal E_g$, the other inclusion being analogous. By the maximality of $\mathcal E_g$ it suffices to show that $g^{-1}(\mathcal E_f) \subset \mathcal E_f$ and by the maximality of $\mathcal E_f$ this will follow if we show that $f^{-1} (g^{-1}(\mathcal E_f)) \subset g^{-1}(\mathcal E_f).$

Let $V$ be an irreducible component of $f^{-1} (g^{-1}(\mathcal E_f))$. We have that $f(V) \subset g^{-1}(\mathcal E_f)$, so $f(g(V)) = g(f(V)) \subset \mathcal E_f$, meaning that $g(V) \subset f^{-1}(\mathcal E_f) = \mathcal E_f$. This shows that $V \subset g^{-1}(\mathcal E_f)$ finishing the proof.
\end{proof}

\subsection*{Lattès maps} Lattès maps (see the Introduction for the definition) form a special class of holomorphic endomorphisms of $\pr^k$. They possess a certain rigidity and therefore can be characterized by several extremal properties, see for instance \cite{berteloot-loeb:lattes} and \cite{berteloot-dupont:lattes}. One of these characterizations is in terms of Green currents and measures.

\begin{theorem} [Berteloot-Loeb, Berteloot-Dupont] \label{thm:m=4-lattes} Let $f: \pr^k \to \pr^k$ be a holomorphic endomorphism of degree $d \geq 2$.
\begin{enumerate}
\item If the Green current of $f$ is smooth and strictly positive in some nonempty open subset of $\pr^k$, then $f$ is a Lattès map.
\item If the Green measure of $f$ is nonzero and absolutely continuous with respect to the Lebesgue measure in some open subset of $\pr^k$, then $f$ is a Lattès map.
\end{enumerate}
\end{theorem}

\begin{corollary} \label{cor:commuting-lattes}
Let $f,g: \pr^k \to \pr^k$ be a commuting pair of endomorphisms of degree $d_f,d_g \geq 2$. If $f$ is a Lattès map then so is $g$.
\end{corollary}
\begin{proof}
Since $f$ is a Lattès map its Green current is smooth and strictly positive in some nonempty open subset of $\pr^k$. From Proposition \ref{prop:same-G} the same holds for the Green current of $g$. The conclusion follows from Theorem \ref{thm:m=4-lattes}.
\end{proof}

\section{Entire curves and Ahlfors currents}

When dealing with non-normal families of meromorphic functions of one variable, a useful tool is the classical Zalcman's renormalization lemma. By applying it to the several variable setting we get the following result.

\begin{lemma} ["Zalcman's Lemma", \cite{dinh-sibony:endo-permutables}]  \label{lemma:zalcman} Let $\{\psi_n\}$ be a non-equicontinuous family of germs of holomorphic maps from a neighborhood of the origin in $\C^k$ taking values in $\C^\ell$. Suppose that the sequence $\{\psi_n (0)\}$ is bounded. Then there is a nonzero vector $v \in \C^k$, a sequence of complex numbers $z_i \to 0$, a sequence of real numbers $\rho_i \to 0$ and an increasing sequence of integers $n_i \to \infty$ such that the maps $\psi_{n_i}(z_i v + \rho_i \xi v)$ converge to a non-constant holomorphic map $\phi(\xi)$ from $\C$ to $\C^\ell$.

\end{lemma}

In the same spirit of renormalization there is another useful tool which is the notion of Ahlfors current. It is a positive closed current constructed from the image of an increasing sequence of discs by an entire map.

\begin{theorem} \label{thm:ahlfors-construction}
Let $\phi: \C \to X$ be a non constant holomorphic map onto a complex manifold $X$. Denote by $D_r$ the closed disc of radius $r$ and center $0$ in $\C$. Then there is a sequence $r_n \nearrow \infty$ such that the positive currents $$R_n = \frac{\phi_* [D_{r_n}]}{\|\phi_* [D_{r_n}]\|}$$ converge to a positive closed current $R$ of mass $1$ whose support is contained  in $\overline {\phi(\C)}$.

\end{theorem}

As an application we have the following.

\begin{proposition} \label{prop:J1=J2}
Let $f:\pr^2 \to \pr^2$ be a holomorphic endomorphism of degree $d \geq 2$. Suppose that the maximal order Julia set $J_2$ contains a holomorphic disc passing through a repelling periodic point. Then $J_1 = J_2$. 
\end{proposition}

\begin{proof}
Let $\Sigma$ be a holomorphic disc passing through a repelling periodic point $a \in J_2$. After replacing $f$ by an iterate we may suppose that $a$ is a fixed point. Consider the Poincaré map $\varphi: \C^k \to \pr^k$ associated with $a$ and let $\Lambda$ be the corresponding  triangular lift (see section \ref{sec:poincare}). Set $\Sigma^* = \varphi^{-1}(\Sigma)$ and $J^* = \varphi^{-1}(J_2)$. Notice that $J^*$ is invariant by $\Lambda$.

  Since $\Lambda$ is dilating, the family $\{\Lambda^n|_{\Sigma^*}: n\geq 1\}$ is not normal. Using Lemma \ref{lemma:zalcman} we can produce an entire curve whose image is contained in $\overline{\bigcup_{n \geq 1}\Lambda^n(\Sigma^*)} \subset J^*$. Composing with $\varphi$ we get an entire curve contained in $J_2$. By Theorem \ref{thm:ahlfors-construction}  there is a positive closed $(1,1)$-current $R$ supported by $J_2$.
  
The sequence of Cesaro sums $S_n = \frac{1}{n} \sum_{j=1}^n d^{-j} (f^j)^*R$ admits a subsequence converging  to a positive closed $(1,1)$-current of unit mass $S$ that is totally invariant $f^*S = d \cdot S$ (Lemma \ref{lemma:cesaro}). Since $J_2$ is closed and totally invariant, the current $S$ is supported by $J_2$, so $S = T$ by Proposition \ref{prop:currentJk}. In particular $ J_1  = \supp S  \subset J_2$. As $J_1 \supset J_2$, the result follows.
\end{proof}

\section{Repelling periodic points and Poincaré maps} \label{sec:poincare}

A germ of holomorphic map $\Lambda:(\C^k,0) \to (\C^k,0)$ is called {\it triangular} if it is invertible and has the form $$ \Lambda(z)  = \big(\lambda_1z_1, \lambda_2 z_2 +  P_2(z_1), \ldots , \lambda_k z_k + P_k(z_1,\ldots,z_{k-1}) \big),$$
where the $P_j$ are polynomials containing only the monomials $z_1^{\ell_1} \cdots z_{j-1}^{\ell_{j-1}}$ for which $\lambda_j = \lambda_1^{\ell_1} \cdots \lambda_{j-1}^{\ell_{j-1}}$. In other words, $\Lambda$ may contain non-linear terms in the presence of resonances of the $k$-uple $\lambda = (\lambda_1,\ldots,\lambda_k)$.

In the case $k=2$ the triangular maps take the form $\Lambda(z) = (\lambda_1 z_1, \lambda_2 z_2 + \alpha z_1^ \ell)$ if $ \lambda_2 = \lambda_1^\ell$ and we call them resonant of index $\ell$, or $\Lambda(z) = (\lambda_1 z_1, \lambda_2 z_2)$ if $\lambda_2 \neq \lambda_1^\ell$ for every $\ell = 1,2,\ldots$ and we call them non-resonant.

The following statements follow from straightforward computations.

\begin{lemma} \label{lemma:lambda-gamma}
Let $\Lambda$ and $\Gamma$ be two triangular maps of $\C^2$.
\begin{enumerate}
\item Suppose that $\Lambda \circ \Gamma = \Gamma \circ \Lambda$ and that one of $\Lambda$ or $\Gamma$ is non-diagonal. Then $\Lambda$ and $\Gamma$ have the same resonance index.

\item Suppose that $\Lambda$ and $\Gamma$ are $\ell$-resonant, so that $\Lambda(z_1,z_2) = (\lambda_1 z_1, \lambda_1^\ell z_2 + \alpha z_1^ \ell)$ and $\Gamma(z_1,z_2) = (\gamma_1 z_1, \gamma_1^\ell z_2 + \beta z_1^ \ell)$. Denoting $\xi = \lambda_1 \gamma_1^{-1}$ we have  $$\Lambda \circ \Gamma^{-1}(z_1,z_2) = (\xi z_1,\xi^\ell z_2 + \gamma_1^{-\ell}(\alpha - \beta \xi^\ell) \, z_1^ \ell)$$ and $$(\Lambda \circ \Gamma^{-1})^n(z_1,z_2) = (\xi^n z_1,\xi^{n\ell} z_2 + n \xi^{(n-1)\ell} \gamma_1^{-\ell}(\alpha - \beta \xi^\ell) \, z_1^ \ell).$$

In particular, the group generated by $\Lambda$ and $\Gamma$ is contained in the two-dimen\-sional complex Lie group of maps of the form $$\Lambda_{t,s} (z_1,z_2) = (t z_1, t^\ell z_2 + s z_1^\ell),  \;\; t \in \C^*, s \in \C.$$
\end{enumerate}

\end{lemma}

We recall a classical theorem of Sternberg.

\begin{theorem*}[Sternberg]
Let $h:(\C^k,0) \to (\C^k,0)$ be a germ of holomorphic map. Suppose that $0$ is a repelling fixed point for $h$. Then there is an invertible holomorphic germ $\varphi: (\C^k,0) \to (\C^k,0)$ such that $\varphi^{-1} \circ h \circ \varphi$ is triangular.
\end{theorem*}

Using the dilation nature of the triangular map it is possible to obtain a global version of the above theorem.

\begin{proposition} \cite{dinh-sibony:endo-permutables} \label{prop:poincare}
Let $f$ be a holomorphic endomorphism of $\pr^k$ of degree $d \geq 2$. Let $b$ be a repelling fixed point of $f$. Then there is a holomorphic map $\varphi: \C^k \to \pr^k$ with $\varphi(0) = b$ and $D \varphi(0)$ invertible and a triangular map $\Lambda$ such that $f \circ \varphi = \varphi \circ \Lambda$.

If furthermore $b \in J_k$ then the complement $\pr^k \setminus \varphi(\C^k)$ is the exceptional set $\mathcal E$ of $f$.
\end{proposition}

\begin{definition}
The map $\varphi: \C^k \to \pr^k$ as in Proposition \ref{prop:poincare} is called the {\bf Poincaré map} associated with $b$.
\end{definition}

The following statement, together with Proposition \ref{prop:same-G} imply that two commuting endomorphisms admit a common Poincaré map.

\begin{proposition} \cite{dinh-sibony:endo-permutables} \label{prop:poincare2}
 Let $f,g:\pr^k \to \pr^k$ be two commuting endomorphisms of degree $d_f,d_g \geq 2$. Let $b$ be a common fixed point for $f$ and $g$ which is repelling for $f$. Then there exist a holomorphic map $\varphi: \C^k \to \pr^k$ with $\varphi(0) = b$ and $D\varphi(0)$ invertible and commuting triangular maps $\Lambda_f, \Lambda_g$ such that $f \circ \varphi = \varphi \circ \Lambda_f$ and $g \circ \varphi = \varphi \circ \Lambda_g$.
\end{proposition}

In the next sections, we will need to apply the above proposition to different periodic points on $J_k$. For this purpose we will the use next lemma together with the fact that the repelling periodic points belonging to $J_k$ are dense in $J_k$ \cite{briend-duval:liapounoff}.

\begin{lemma}  \cite{dinh-sibony:endo-permutables} \label{lemma:common-periodic}
Let $f,g:\pr^k \to \pr^k$ be a commuting pair of endomorphisms. Then $f$ and $g$ possess infinitely many common periodic points in $J_k$ which are repelling for $f$. Furthermore, given a repelling periodic point $a$ of $f$ there is an $m\geq 0$ such that $g^m(a)$ is periodic for $f$ and $g$ and repelling for $f$.
\end{lemma}

\section{Laminations of the equilibrium measure}

In this section, we follow \cite{dinh-sibony:endo-permutables} and show that the existence of the maps $\Lambda_f$ and $\Lambda_g$ as in Proposition \ref{prop:poincare2} enables us to produce laminations of the Julia set and  of the measure $\mu$.

We start by recalling a notion of laminarity that will be useful to us. It is worth mentioning that there are several non-equivalent definitions of this concept.

\begin{definition}
Let $X$ be a real analytic manifold of real dimension $n$. Let $J \subset X$ be a closed subset and $b \in J$.   We say that $J$ is \textit{$m$-laminated} at $b$ ($0\leq m \leq n$) if there is a neighborhood of $b$ of the form $ \Omega = U \times V$ with $U \subset \R^m$ and $V \subset \R^{n-m}$ open subsets and a real analytic coordinate system $(x_1,\ldots,x_m, x_{m+1}, \ldots, x_n)$ defined on $\Omega$ such that $\Omega \cap J = U \times K$ for some closed set $K \subset V$.

Let $\nu$ be a measure on $X$. We say that $\nu$ is \textit{$m$-laminated} at $b$ if there are $\Omega$ and $(x_1,\ldots,x_m, x_{m+1}, \ldots, x_n)$ as above such that $\nu$ is a product of the Lebesgue measure $\Leb_m$ on $U$ and a measure $\nu'$ on $V$.
\end{definition}

\begin{remark} \label{rmk:local-flow}
Notice that if a measure $\nu$ is $m$-laminated at a point $b$ then the constant vector fields $\del / \del x_1, \ldots, \del / \del x_m$ on $\R^n$ induce vector fields near $b$ whose corresponding local flows preserve $\nu$.
\end{remark}

The following lemma is an adaptation of a result by Eremenko (\cite{eremenko:functional-equations}, Proposition 1).

\begin{lemma} \label{lemma:eremenko}
Let $\nu$ be an $m$-laminated measure of the form $\nu = \Leb_m \otimes \nu'$ on $\Omega = U \times V \subset \R^m \times \R^{n-m}$. Let $\vartheta: \Omega \to \R^{n}$ be a non-vanishing vector field and denote by $\Gamma^t$ its local flow. If $(\Gamma^t)^* \nu = \nu$ and if $\vartheta(p)$ is transversal to $\R^m$, $p \in \Omega$ then $\nu$ is $(m+1)$-laminated at $p$.
\end{lemma}




The main result of this section is a description of the local structure of the maximal order Julia set and the Green measure of a commuting pair of same degree.

\begin{proposition} \label{prop:lamination}
Let $f,g: \pr^2 \to \pr^2$ be commuting endomorphisms of degree $d$ such that $f^n \neq g^n$ for all $ n \geq 1$. Then the maximal order Julia set $J := J_2(f) = J_2(g)$ and the equilibrium measure  $\mu = \mu_f = \mu_g$ are laminated outside an analytic subset of $\pr^2$. If $m$ denotes the maximal integer such that $J$ and $\mu$ are $m$-laminated in some open subset of $\pr^2$, then $m=3$ or $4$.
\end{proposition}

Before proceeding to the proof let us introduce some notation and state some preliminary results.

Fix a common repelling fixed point $a$ of $f$ and $g$. Let $\varphi:\C^2 \to \pr^2$ be its associated Poincaré map and $\Lambda_f,\Lambda_g$ the corresponding triangular lifts. Define the current $T^* = \varphi^* T$ and the measure $\mu^* = \varphi^* \mu$. When $f$ and $g$ are polynomial, set $G^* = G \circ \varphi$, where $G$ is the common Green function of $f$ and $g$.

We have the following invariance relations: $$\Lambda_f^* T^* = d \cdot T^* =\Lambda_g T^*, \;\;\;\;\Lambda_f^* \mu^* = d^2 \cdot \mu^* = \Lambda_g^* \mu^*$$ and when $f$ and $g$ are polynomial we have also $G^* \circ \Lambda_f= d\cdot G^* = G^* \circ \Lambda_g$.

\begin{lemma} \label{lemma:generated-group}
The closed group generated by $\Lambda_f \circ \Lambda_g ^{-1}$, i.e.\ the closure of $\{\Lambda_f^n \circ \Lambda_g ^{-n}\}_{n \in \N}$ in the group of triangular maps, is a compact, positive dimensional Lie group. 
\end{lemma}

\begin{proof}
Notice that the $\Lambda_f^n \circ \Lambda_g ^{-n}$ are pairwise distinct by condition (\ref{eq:disjoint-iterates}). The maps $\Lambda_f$ and $\Lambda_g$ belong to a two dimensional complex Lie group composed by maps of the form $\Lambda_{t,s} (z)= (t z_1, t^\ell z_2 + s z_1^\ell)$ if there are resonances or $\Lambda_{t,s} = (t z_1, s z_2)$ otherwise (see Lemma \ref{lemma:lambda-gamma}).

We claim that sequence $\{\Lambda_f^n \circ \Lambda_g ^{-n}\}$ is  bounded. Suppose that this is not the case, so the coefficients of $\Lambda_f^n \circ \Lambda_g ^{-n}$ are unbounded. In particular the family $\{\Lambda_f^n \circ \Lambda_g ^{-n}\}$ is not equicontinuous. From Lemma \ref{lemma:zalcman} we can find a sequence of integers $\{n_i\}$, sequences $\{z_i\}$ and $\{\rho_i\}$ converging to $0$ and a vector $v$ in $\C^2$ such that $\Lambda_f^{n_i} \circ \Lambda_g ^{-n_i}(z_i v + \rho_i \zeta v)$ converges to a non-constant holomorphic map $h(\zeta)$ from $\C$ to $\C^2$.

We claim that $h^*T^* = 0$. Set $\phi_i (\zeta) = z_i v + \rho_i \zeta v$ and $h_i = \Lambda_f^{n_i} \circ \Lambda_g ^{-n_i} \circ \phi_i$, so $h = \lim h_i$. As $(\Lambda_f^{n_i} \circ \Lambda_g ^{-n_i})^* T^* = T^*$ we have $h_i^*T^* = \phi_i^* T^*$. It suffices then to show that $\phi_i^* T^* \to 0$ as $i \to \infty$. Writing $T^* = \ddc u$ for some continuous p.s.h.\ function on $\C^2$ with $u(0) = 0$, it is easy to see that $u \circ \phi_i \to 0$ in $L^1_{loc}$, as the image of a fixed compact set by the $\phi_i$  tends to $0$. This yields $\phi_i^* T^* \to 0$ proving the claim.

Write $T = \omega + \ddc \mathfrak g$, where $\omega$ is the Fubini-Study form and $\mathfrak g$ is a bounded function. Since $h^* \varphi^* T = h^*T^* = 0$ we have $\ddc (-\mathfrak g \circ \varphi \circ h) =  h^* \varphi^* \omega \geq 0$ on $\C$. In particular $-\mathfrak g \circ \varphi  \circ h $ is a bounded subharmonic function on $\C$, so it must be constant. We have then that $h^* \varphi^* \omega = 0$, which is absurd since the $(1,1)$-form $\omega$ is strictly positive.

The above contradiction shows that the sequence $\{\Lambda_f^n \circ \Lambda_g ^{-n}\}$ is bounded. 
As it is infinite, it must have an accumulation point. In particular, its closure is a compact, positive dimensional Lie group.
\end{proof}

\begin{corollary} \label{cor:abs-values}
Let $\lambda_i$ and $\gamma_i$, $i=1,2$ be the eigenvalues of $\Lambda_f'(0)$ and $\Lambda_g'(0)$ respectively, ordered by their absolute value. Then $|\lambda_i|=|\gamma_i|$ for $i=1,2$.
\end{corollary}
\begin{proof}
From the above lemma the sequence $\{ \lambda_i^n \gamma_i ^{-n}\}$ is bounded. Let $\xi$ be one of its cluster points and $n_j \nearrow + \infty$ such that $\lambda_i^{n_j} \gamma_i ^{-n_j} \to \xi$. Since $\xi^k$, $k = \pm  1, \pm 2, \ldots$ are also cluster points and they should be bounded we must have $|\xi| = 1$. We thus have $| \lambda_i^{n_j} \gamma_i ^{-n_j} | \to 1$ which is possible only if $|\lambda_i|=|\gamma_i|$.
\end{proof}

\begin{lemma} \label{lemma:mgeq2}
Let $m$ be the maximal integer such that $\mu$ is $m$-laminated in some open set of $\pr^2$. Then $m \geq 2$.
\end{lemma}

\begin{proof}
Pick a common fixed point $a \in J_2$ which is repelling for $f$. Consider the associated Poincaré map $\varphi$ and the triangular lifts $\Lambda_f$, $\Lambda_g$. By Lemma \ref{lemma:generated-group}  there is a $1$-parameter group of automorphisms preserving $\mu^*$ and $J^*$. This shows that the set $J^*$ and the measure $\mu^*$ are at least $1$-laminated. By pushing the lamination forward by $\varphi$ we see that $J$ and $\mu$ are at least $1$-laminated outside the set of critical values of $\varphi$.

Let $m$ be as in the statement of the lemma. Notice that each leaf of the $m$-lamination of $J^*$ is  invariant by the Lie group $\mathcal G$ generated by $\Lambda_f \circ \Lambda_g ^{-1}$, where $\Lambda_f$ and $\Lambda_g$ are the triangular lifts of $f$ and $g$ at some common repelling periodic point. Indeed, if the orbit under $\mathcal G$ of some point in $J^*$ connects two different leaves, the fact that $\mathcal G$ is positive dimensional together with Lemma \ref{lemma:eremenko} would allow us to produce an $(m+1)$-lamination of $\mu$ in some open set, contradicting the maximality of $m$.

Let us show that $m \geq 2$. Suppose that $m=1$.  Using Lemma \ref{lemma:common-periodic} and the density of repelling periodic points in $J_2$ we may suppose, after replacing $f$ and $g$ by some iterates, that the leaf of $J_2$ passing through $a$ is smooth. We claim that $\Lambda_f$ and $\Lambda_g$ are diagonal. If none of them is resonant the result follows from Sternberg's Theorem, so we may suppose that $\Lambda_f$ is resonant.

 By assumption, there is a real analytic curve $C_0$ contained in $J^*$ passing through $0$. Notice that $C_0$ is invariant by $\Lambda_f$ and $\Lambda_g$. Indeed, if this were not the case then the image of  $C_0$ by $\Lambda_f$ would cut the nearby leaves transversally, but then using the invariance of $\mu^*$ by $\Lambda_f$ we would get a $2$-lamination of $\mu^*$ at the intersection point, contradicting the maximality of $m$. The same reasoning applies to $\Lambda_g$.
 
 After a linear change of coordinates that does not change the triangular form of $\Lambda_f$ and $\Lambda_g$ we may assume that $L =T_0 C_0 = \{ z_1 =  0, \im z_2 =0 \}$. Since $C_0$ is invariant by $\Lambda_f$ and $\Lambda_g$ we get that $L$ by is invariant by $\Lambda_f'(0)$ and $\Lambda_g'(0)$. In particular $\lambda_2$ and $\gamma_2$ are real. By Corollary \ref{cor:abs-values} we get $\lambda_2 = \pm \gamma_2$. Replacing  $\Lambda_f$ and $\Lambda_g$ by $\Lambda^2_f$ and $\Lambda^2_g$ if necessary we may assume that  $\lambda_2 = \gamma_2$.

If $\Lambda_f$ and $\Lambda_g$ are both resonant, the relation $|\lambda_1| = |\gamma_1|$ shows that they have the same resonance index, so $\lambda_1^\ell = \lambda_2 = \gamma_2 = \gamma_1^\ell$ for some $\ell \in \N$. Replacing  $\Lambda_f$ and $\Lambda_g$ by their $\ell$-th iterate we may assume that  $\lambda_1 = \gamma_1$, so they have the form $\Lambda_f (z_1,z_2) = (\lambda z_1, \lambda^\ell z_2 + \alpha z_1^\ell)$ and $\Lambda_g (z_1,z_2) = (\lambda z_1, \lambda^\ell z_2 + \beta z_1^\ell)$. In this case $\Lambda_f \circ \Lambda_g^{-1} = (z_1, z_2 + \lambda^{-\ell}(\alpha - \beta)z_1^\ell)$ generates an unbounded group unless $\alpha = \beta$ (see Lemma \ref{lemma:lambda-gamma}). But then  $\Lambda_f = \Lambda_g$, contradicting condition (\ref{eq:disjoint-iterates}). This contradiction shows that  $\Lambda_g$ cannot be resonant. By the commutativity of $\Lambda_f$ and $\Lambda_g$  we get that $\Lambda_f$ is diagonal.

By the above discussion $\Lambda_f$ and $\Lambda_g$ are diagonal with $|\lambda_1| = |\gamma_1|$ and $\lambda_2 = \gamma_2$, so the automorphism group generated by  $\Lambda_f \circ  \Lambda_g^{-1}$ must be of the type $\Lambda^t(z_1,z_2) = (e^{it}z_1,z_2)$. If we choose a different repelling periodic point $a'  \in J_2 $ in $C_0$ and consider the corresponding Poincaré map $\varphi'$ and the triangular lifts $\Lambda'_f,\Lambda'_g$ we can repeat the above arguments and get another group $(\Lambda')^t$ of  the type above, whose orbits will be transversal to the orbits of $\Lambda^t$. This will produce a $2$-lamination of $\mu^*$ at some point, contradicting the maximality of $m$.
\end{proof}

\begin{lemma} \label{lemma:S0totreal}
Let $m$ be the maximal integer such that $\mu$ is $m$-laminated in some open set $\Omega$ of $\pr^2$ and suppose that $m=2$. Then, after reducing $\Omega$ if necessary, all the leaves of the lamination are totally real analytic surfaces.
\end{lemma}

\begin{proof}
Pick a common repelling periodic point $a \in J_2$ at which the leaf passing through it is smooth (we can do so using by Lemma \ref{lemma:common-periodic} and the density of repelling periodic points in $J_2$). After replacing $f$ and $g$ by some iterate we may suppose that $a$ is a common fixed point. Let $\varphi:\C^2 \to \pr^2$ be the associated Poincaré map, $\Lambda_f$ and $\Lambda_g$ the corresponding triangular lifts and $J^* = \varphi^{-1} (J_2)$. Denote by $S_c$ the leaves of the lamination of $J^*$. 

If there is one leaf that is totally real then, by continuity, the nearby leaves are also totally real (it may as well happen that there is only one leaf). In this case we can shrink $\Omega$ so it contains only totally real leaves.

Suppose by contradiction that all the leaves are complex curves and denote by $S_0$ the leaf of $J^*$ through the origin in $\C^2$. Notice that there is no open subset $U$ of $\C^2$ containing only one leaf, otherwise $J^* \cap U$ would be pluripolar. Therefore there infinitely many leaves accumulating around $S_0$.

If $\Lambda_f$ and $\Lambda_g$ are both resonant,  the relations $|\lambda_i| = |\gamma_i |$, $i=1,2$ imply  that they have the same resonance index, say $\ell$. Writing $\Lambda_1(z_1,z_2) = (\lambda_1 z_1, \lambda_1^\ell z_2 + \alpha z_1^ \ell)$ and $\Lambda_2(z_1,z_2) = (\gamma_1 z_1, \gamma_1^\ell z_2 + \beta z_1^ \ell)$, Lemma \ref{lemma:lambda-gamma} shows that  $\Lambda_f \circ \Lambda_g ^{-1}$ generate an unbounded group unless $\alpha = \beta \xi^\ell$, where $\xi = \lambda_1 \gamma_1^{-1}$ and in this case the Lie group they generate contains a $1$-parameter subgroup of the form $\Lambda^t(z_1,z_2) = (e^{it}z_1,e^{i \ell t} z_2)$. It is easy to see, by looking at the equation defining them, that the only complex curves invariant by this group are the axes $\{z_1 = 0\}$,  $\{z_2 = 0\}$ and $z_2 = \kappa z_1^\ell$, $\kappa \in \C$. Since all these curves pass through $0$, the leaves $S_c$ with $c \neq 0$ small must cut them transversally, which is not possible by Lemma \ref{lemma:eremenko} and the maximality of $m$.

The above discussion shows that one of $\Lambda_f$ or $\Lambda_g$ is non-resonant. The commutativity of $\Lambda_f$ and $\Lambda_g$ implies that they are both diagonal. The Lie group generated by $\Lambda_f \circ \Lambda_g^{-1}$ contains then a $1$-parameter subgroup of one of the following types (i) $\Lambda^t(z_1,z_2) = (e^{it}z_1, z_2)$, (ii) $\Lambda^t(z_1,z_2) = (z_1, e^{it} z_2)$ or (iii) $\Lambda^t(z_1,z_2) = (e^{ipt}z_1, e^{iqt} z_2)$ $p,q, \in \N$ or the $2$-dimensional group $\Lambda^{s,t}(z_1,z_2) = (e^{is}z_1, e^{it}z_2)$. 

It is clear that there is no complex curve invariant by $\Lambda^{s,t}$, so case (iv) cannot occur. For the group in case (iii) the only invariant complex curves are $z_2^p = \kappa z_1^q$, $\kappa \in \C$ and the coordinate axes. Since they all pass through $0$ we argue as above to see that this contradicts the maximality of $m$, so case (iii) cannot occur either. For the group in case (i) it is easy to see that the only invariant holomorphic curves are of the form $\{z_2 = c\}$ for some constant $c$ and similarly for the case (ii), replacing $z_2$ by $z_1$. We conclude that, up to permuting $z_1$ and $z_2$ we have $S_c = \{z_2 = c\}$, so
\begin{equation} \label{eq:foliated-julia}
J^* = \bigcup_{c \in K} \{z_2 = c\}
\end{equation}
for some closed set of parameters $K \subset \C$.

Consider now the groupoid
\begin{equation*}
\mathcal A = \left \{ \tau : \,  \begin{split}  &\tau: W \to \tau(W) \text{ is a local biholomorphism} \\ &\text{ of }\C^2 \text{ such that } \varphi\circ \tau = \varphi \text{ and } W \cap J^* \neq \varnothing \end{split} \right \} .
\end{equation*}
Notice that the elements $\tau \in \mathcal A$ preserve $\mu^*$ because  $\tau^* \mu^* = \tau^* \varphi ^* \mu = \varphi ^* \mu= \mu^*$. This implies that they also preserve the leaves of $J^*$. Indeed, suppose that there is a $c \in K$ such that $\tau(S_c)$ is different from any other $S_{c'}$, $c' \in K$. Since there is no isolated leaf we can see that $\tau(S_c)$ cuts a generic leaf transversally, but then Lemma \ref{lemma:eremenko} would allow us to produce a 3-lamination of $\tau^* \mu^* = \mu^*$ on an open set, contradicting the maximality of $m$.

Denote by $\mathcal F^*$ the foliation $\{z_2 = constant.\}$. We claim that an element $\tau: W \to \tau(W)$ in $\mathcal A$ takes a leaf of $\mathcal F^*$ intersecting $W$ to another leaf of $\mathcal F^*$. To see this we write $\tau(z_1,z_2) = (\tau_1(z_1,z_2),\tau_2(z_1,z_2) )$. From the fact that $\tau$ preserves the leaves of $J^*$ we get that $\frac{\del \tau}{ \del z_1}$ vanishes over $J^* \cap W$ so it must be identically zero on $W$ by the identity principle (recall that $K$ has no isolated points). This shows that $\tau_2$ is independent of $z_1$, proving the claim.

Let $A \subset \pr^2$ be the complement of the image of the regular values of $\varphi$. This is a pluripolar set containing $\mathcal E$ and the above discussion shows that $\varphi$ induces a nonsingular holomorphic foliation $\mathcal F$ on a neighborhood of $J_2 \setminus A$: the leaf of $\mathcal F$ through $b$ is defined as the image by $\varphi$ of the leaf of $\mathcal F^*$  through $w$, where $w$ is any regular value in $\varphi^{-1}(b)$. From (\ref{eq:foliated-julia}) we have that, outside $A$, $J_2$ is locally given by a union of leaves of $\mathcal F$ and that these leaves are images of $\C$. Lemma \ref{lemma:no-hol-fol} below gives a contradiction.
\end{proof}

 Let $\mathcal F$ be a non-singular foliation by Riemann surfaces on a complex surface $X$. We say that a $(1,1)$-current $S$ on $X$ is directed by $\mathcal F$ if $S \wedge \gamma = 0$, where $\gamma$ is a local holomorphic $(1,0)$-form whose kernel defines $\mathcal F$. The following result is well known, see for instance \cite{dujardin-guedj:maximal}.
 
 \begin{lemma}
 Let $S$ be a  $(1,1)$-current with continuous local potentials on a complex surface. If $S$ is directed by a non-singular foliation then $S \wedge S = 0$.
 \end{lemma}

 


\begin{lemma} \label{lemma:no-hol-fol}
Let $f:\pr^2 \to \pr^2$ be a holomorphic endomorphism of degree $d \geq 2$. Suppose that there is a pluripolar set $A$ and a nonsingular $f$-invariant  holomorphic foliation $\mathcal F$ on an open neighborhood $\mathcal W$ of $J_2 \setminus A$ such that, outside $A$,  $J_2$ is locally a union of leaves of $\mathcal F$. Then there is no entire curve $\phi:\C \to J_2$ whose image is contained in a leaf of $\mathcal F$.
\end{lemma}

\begin{proof}
Let $L$ be a leaf of $\mathcal F$ contained in $J_2$ and suppose that there is a non-constant holomorphic map $\phi:\C \to L$. By the Ahlfors construction there is a positive closed $(1,1)$-current $R$ supported on $\overline{\phi(\C)} \subset J_2$. Taking a subsequence of the Cesaro sums $S_n = \frac{1}{n} \sum_{j=1}^n d^{-j} (f^j)^*R$  we get a totally invariant current $S$ supported by $J_2$. By Proposition \ref{prop:currentJk} we have $S=T$. In particular $S$ has continuous potentials,  $S \wedge S$ is well-defined and does not charge pluripolar sets. 

Notice that, by construction, the current $R$ is directed by $\mathcal F$ on $\mathcal W$. From the invariance of $\mathcal F$ by $f$, the currents $S_n$ are also directed by $\mathcal F$ on  $\mathcal W$, and the same holds for the limit value $S$. In particular $S \wedge S = 0$  on $\mathcal W$, so $S \wedge S$ is supported by $A$. This gives a  contradiction, since $A$ is pluripolar.
\end{proof}

\begin{proof}[Proof of Proposition \ref{prop:lamination}] 
We know from Lemma \ref{lemma:mgeq2} that $m \geq 2$. Suppose that $m=2$. We keep the notation as in the proof of Lemma \ref{lemma:S0totreal} and denote by $S_0$ be the leaf through $0$. By Lemma \ref{lemma:S0totreal}, $S_0$ is a totally real surface. Using Lemma \ref{lemma:eremenko}, the fact that $\Lambda_f$ and $\Lambda_g$ fix the origin and the maximality of $m$ we see that $S_0$ is invariant by $\Lambda_f$ and $\Lambda_g$. After a linear change of coordinates that does not change the triangular form of $\Lambda_f$ and $\Lambda_g$ we may assume that $H =T_0 S_0 = \{\im z_1 =  \im z_2 =0 \}$. By the invariance of $S_0$, $H$ is invariant by $\Lambda_f'(0)$ and $\Lambda_g'(0)$. In particular the eigenvalues $\lambda_i,\gamma_i$, $i=1,2$ are real. From Corollary \ref{cor:abs-values} we have $\gamma_i = \pm\lambda_i$. In particular $\Lambda_f^2 \circ \Lambda_g ^{-2} (z_1,z_2) = (z_1,z_2 + \delta z_1^\ell)$ cannot generate a non-trivial compact group unless $\delta = 0$ (Lemma \ref{lemma:lambda-gamma}). In this case we have $\Lambda_f^2 = \Lambda_g^2$, which cannot happen in view of condition (\ref{eq:disjoint-iterates}). This contradiction shows that $m \geq 3$.
\end{proof}

\section{Laminations by hypersurfaces} \label{sec:m=3}

In this section we will describe the commuting pairs of endomorphisms of $\pr^2$ such that the corresponding Green measure is $3$-laminated. We will see later (Theorem \ref{thm:3-laminated-polynomial}) that the maps in question are both polynomial after some number of iterations. For this reason we begin by analyzing the polynomial case.

\subsection{Polynomial case}
Let $f,g: \pr^2 \to \pr^2$ be a commuting pair of holomorphic endomorphisms of same degree $d_f = d_g = d  \geq 2$ satisfying (\ref{eq:disjoint-iterates}).  As before we denote by $\varphi: \C^2 \to \C^2$ the Poincaré map associated with a common repelling fixed point $a \in J_2$ and $\Lambda_f, \Lambda_g$ the corresponding triangular lifts.

\begin{lemma} \label{lemma:circle-group}
Let $m$ be the maximal integer such that $\mu$ is $m$-laminated in some open set of $\pr^2$. Suppose that $m=3$. Then
\begin{enumerate}
\item The maps  $\Lambda_f$ and $\Lambda_g$ are linear and diagonal. In particular, the axes $\{z_1 = 0\}$ and $\{z_2 = 0\}$ are invariant.
\item The group $\mathcal G$ generated by $\Lambda_f \circ \Lambda_g^{-1}$ is one-dimensional and after a linear change of coordinates of $\C^2$ the connected component of the identity $\mathcal G_0$ contains a subgroup of the form $\Lambda^t = \bigl( \begin{smallmatrix} e^{it} & 0 \\ 0 & 1 \end{smallmatrix} \bigr)$.
\end{enumerate}
\end{lemma}

\begin{proof}

Let $S_0$ the leaf through $0$ coming from the lamination of $\mu^*$. It is a real analytic hypersurface and since $\Lambda_f$ and $\Lambda_g$ preserve $\mu^*$ up to a constant factor, $S_0$ is invariant by $\Lambda_f$ and $\Lambda_g$. Indeed, if $S_0$ is the only leaf then it coincides with the support of $\mu^*$, so it must be preserved by $\Lambda_f$ and $\Lambda_g$. If there is another leaf different from $S_0$, there are infinitely many leaves accumulating around $S_0$. In this case, if $\Lambda_f (S_0)$ or $\Lambda_g (S_0)$ is different from $S_0$ then they must cut the nearby leaves transversally, which is not possible by Lemma \ref{lemma:eremenko}.

After a linear change of coordinates that does not change the triangular form of $\Lambda_f$ and $\Lambda_g$ we may assume that $H =T_0 S_0 = \{ \im z_2 =0 \}$. Since $\Lambda_f$ and $\Lambda_g$ fix $0$ and leave $S_0$ invariant, their derivatives preserve $H$. This implies that  $\lambda_2$ and $\gamma_2$ are real. Since $|\lambda_2| = |\gamma_2|$ from Corollary \ref{cor:abs-values}, we have that $\lambda_2 = \pm \gamma_2$.  We  can now follow the arguments as in the the proof of Lemma  \ref{lemma:mgeq2} and conclude.
\end{proof}

The following result is implicitly contained in \cite{dinh:lattes}. 

\begin{theorem} [Dinh] \label{thm:dinh-lattes}
Let $f$ be a polynomial endomorphism of degree $d \geq 2$ of $\C^k$ that extends to a holomorphic endomorphism of $\pr^k$. Suppose that there is an open set $\Omega$ such that $J_k \cap \Omega$ is a real analytic hypersurface of $\Omega$. Then there is a repelling periodic point $a \in J_k \cap \Omega$ at which $J_k$ is strictly pseudoconvex and  there is a suitable coordinate system of $\C^k$ in which the Poincaré map associated with $a$ has the form $\varphi(z',z_k) = \nu_1(z') \exp(2 \pi i z_k)$ for some $\nu_1: \C^{k-1} \to \C^k$. Furthermore, the map $f$ is homogeneous and the restriction of $f$ to the hyperplane $\pr^k \setminus \C^k \simeq \pr^{k-1}$ is a Lattès map.  \end{theorem}

The above result shows that, in order to prove Theorem \ref{thm:main-thm} in the polynomial case, it suffices to show that in some open set the lamination of $J_2$ contains only one leaf. This was also done in  \cite{dinh:lattes}. We give here however an alternative proof in the setting of commuting pairs of $\C^2$, since some of the arguments involved will be used later.

\begin{proposition} \label{prop:pol-one-leaf}
Let $f,g: \C^2 \to \C^2$ be commuting polynomial endomorphisms of degree $d$ that extend holomorphically to $\pr^2$. Suppose that the maximal integer $m$ such that $J_2 = J_2(f) = J_2(g)$ is $m$-laminated in some open set is $m=3$. Then there is an open set $\Omega$ such that $J \cap \Omega$ is an analytic hypersurface. In other words, there is only one leaf of $J_2$ intersecting $\Omega$.
\end{proposition}

\begin{proof}
Let $\Omega$ be an open set where $J_2$ is $3$-laminated. Notice that, after shrinking $\Omega$ if necessary, we may assume that every leaf in $\Omega$ is a strictly pseudoconvex hypersurface. Indeed, if there is one leaf in $\Omega$ that is stricly pseudoconvex at some point, then the nearby leaves are strictly pseudoconvex by continuity. If this is not the case then all the leaves in $\Omega$ are Levi-flat hypersurfaces, which are foliated by complex curves. In particular $J_2$ would contain a complex curve passing through a repelling periodic point, which is impossible (Proposition \ref{prop:complex-curves-J}).

Fix a common repelling periodic point $a \in \Omega \cap J_2$. We may replace $f$ and $g$ by suitable iterates and assume that $a$ is fixed. Let $\varphi:\C^2 \to \pr^2$ be the associated Poincaré map and  $\Lambda_f$ and $\Lambda_g$ the corresponding triangular lifts. Similar arguments as in the proof of Lemma \ref{lemma:circle-group} show that $\Lambda_f$ and $\Lambda_g$ preserve the leaves of $J^*$.

\begin{lemma} \label{lemma:disjoint-spheres}
Up to a holomorphic change of coordinates that does not change the diagonal form of $\Lambda_f, \Lambda_g$, the leaves are given by $S_c = \{\im z_2 = |z_1|^2 + c \}$, where $c$ is the leaf passing through $(0,ic)$. Furthermore $\lambda_2 = |\lambda_1|^2$ and  $\gamma_2 = |\gamma_1|^2$.
\end{lemma}

\begin{proof}
Keeping the notation of Lemma \ref{lemma:circle-group} we let $S_0$ be the leaf through $0$, fixing coordinates such that $H = T_0 S_0 = \{\im z_2 = 0\}$ and such that $\Lambda_f$ and $\Lambda_g$ are diagonal. Choose $t$ such that $\Lambda := \Lambda_f \circ \Lambda^t = \bigl( \begin{smallmatrix} \kappa & 0 \\ 0 & \lambda_2 \end{smallmatrix} \bigr) $ with $\kappa,\lambda_2 \in \R$.

We will first show that $S_0$ has the desired form. Locally, we can write $S_0$ as a graph $\im z_2 = h(\re z_1, \im z_1, \re z_2)$, where $h$ is a real analytic function vanishing to order two at $0$. Since $S_0$ is strictly pseudoconvex at $0$, the Taylor expansion of $h$ contains a nontrivial degree two homogeneous polynomial $P_2(z_1, \overline{z}_1)$ as one of its terms. The invariance of $S_0$ by $\Lambda$ gives $h \circ \Lambda = \lambda_2 \im z_2 =  \lambda_2 h$. Comparing each term of the Taylor expansion of $h$ we can see that $\kappa^2 = \lambda_2$ and that $P_2(z_1, \overline{z}_1)$ is the only nonzero term. Therefore $S_0$ is given by $\im z_2 = P_2( z_1,  \overline z_1) = \alpha z_1^2 + \beta |z_1|^2 + \delta \, \overline z _1 ^2$. The invariance of $S$ by $\Lambda^t$ implies that $\alpha = \delta = 0$ and the fact that $S_0$ is strictly pseudoconvex guarantees that $\beta \neq 0$. Therefore, after the change of coordinates $z_1 \mapsto z_1 / \sqrt \beta$ the leaf $S_0$ is given by $\im z_2 = |z_1|^2$.

In order to clarify the notation let us denote $\lambda := \lambda_2$. From the above paragraph $\kappa = \sqrt \lambda$, so $\Lambda = \bigl( \begin{smallmatrix} \sqrt \lambda & 0 \\ 0 & \lambda \end{smallmatrix} \bigr) $.

Since $J^*$ is $3$-laminated there exists a neighborhood of $0$ such that the $S_c$ form a disjoint union of graphs given by equations of the form $\im z_2 = h_c(\re z_1, \im z_1, \re z_2)$, where $c$ varies on some closed subset $K \subset \R$. Here, the functions $h_c$ are real analytic and depend analytically on $c$. In particular the coefficients of $h_c$ are locally bounded uniformly in $c$. The value $c$ is determined by the intersection of $S_c$ with the real line $\{z_1 = 0, \re z_2 = 0\}$, that is, $c = h_c(0)$. From the above discussion we know that $h_0 ( \re z_1, \im z_1, \re z_2 ) = |z_1|^2$.

Since $\Lambda$ takes $(0,ic)$ to $(0,i\lambda c)$,  it maps $S_c$ to $S_{\lambda c}$. This implies that
\begin{equation} \label{eq:h-invariance}
h_{\lambda c} (\re z_1, \im z_1, \re z_2) = \lambda h_c(\lambda^{-1/2} \re z_1,\lambda^{-1/2} \im z_1, \lambda^{-1} \re z_2).
\end{equation}

Consider the Taylor expansion $$h_c(\re z_1, \im z_1, \re z_2) = c+ P_{2,c}(\re z_1, \im z_1) + P_{1,c}(\re z_1, \im z_1) \re z_2 + \alpha_c (\re z_2)^2 + \text{h.o.t},$$ where $P_{j,c}$ are homogeneous of degree $j$. Equation (\ref{eq:h-invariance}) implies that $P_{1,c/ \lambda} = \sqrt \lambda P_{1,c}$.  Iterating we get $P_{1,c/ \lambda^n} =\lambda^{n/2} P_{1,c}$ for every $n \geq 0$. This shows that $P_{1,c} = 0$ for every $c$, otherwise the coefficients of  $P_{1,c}$ would go to infinity when $c$ approaches $0$, which is absurd since the $h_c$ depend analytically on $c$. A similar reasoning shows that all other higher order terms must vanish. Therefore, the only term left is $P_{2,c}$. Since $S_c$ is invariant by $\Lambda ^t$, we have $P_{2,c}(z_1) = \gamma(c) |z_1|^2$ for some continuous function $\gamma$. Using  equation (\ref{eq:h-invariance}) once again we get $\gamma(c) = \gamma(c / \lambda^{n})$ for every $c$. Making $n \to \infty$ we see that  $\gamma$ is constant and equal to $\gamma(0) = 1$. This proves that the $S_c$ have the desired form.
\end{proof}

We are now in position to finish the proof. Observe that $J^* \cap \{z_2 = 0\} = \{0\}$. Indeed, if there exists $(\xi,0) \in J^* \cap \{z_2 = 0\}$ with $\xi \neq 0$, the invariance of $J ^*$ by $\Lambda^t$ implies that $J^* \cap \{z_2 = 0\}$ contains the circle $\{|z_1|= |\xi| \}  \times \{0\}$. In particular the non-negative subharmonic function $G^*(z_1,0)$ vanishes for $|z_1| = |\xi|$ and by the maximum principle it vanishes for $|z_1| \leq |\xi|$. The invariance of $G^*$ by $\Lambda_f$ implies that $G^*(z_1,0) \equiv  0$ and hence $G$ vanishes over the entire curve $\varphi(\{z_2 = 0\})$, yielding a contradiction, since the set $\{G = 0\}$ is compact.

We  now claim that $S = S_0$ is the only leaf of $J^*$. Suppose that $J^*$ contains a leaf $S_c$ with $c>0$ and let $w = (0, ic) \in S_c$.  Consider the cylinders  $C_t = \{ \im z_2 = t \} \cap S_0 = \{(\xi, x + it): |\xi|^2 = t , x \in \R\}$ and their inner sides $D_t = \{(\xi, x+ it): |\xi|^2 < t, x \in \R \}$. The $C_t$ are non-empty for $t \in (c-\varepsilon, c+ \varepsilon)$ if $\varepsilon > 0$ is small. In particular $G^*$ vanishes over $C_t$ and by the maximum principle it must vanish on $D_t$. We conclude that $G^*$ vanishes over $W := \bigcup_{t \in (c-\varepsilon, c+ \varepsilon)} D_t$, which is an open neighborhood of $w$. This contradicts the fact that $J^* \subset \del \{G^* = 0\}$.

Suppose now that $J^*$ contains a leaf $S_c$ with $c<0$. Consider $C'_t = \{ \im z_2 = t \} \cap S_c = \{(\xi, x + it): |\xi|^2 + c = t , x \in \R\}$ and $D'_t = \{(\xi, x+ it): |\xi|^2 + c < t, x \in \R \}$ for $t \in (-\varepsilon, \varepsilon)$. Using the fact that $G^*=  0 $ on $S_c$ and the maximum principle we see that $G^*$ vanishes on the neighborhood of the origin $W' := \bigcup_{t \in (-\varepsilon, \varepsilon)} D'_t$. This is impossible since $0 \in J^* \subset \del \{G^* = 0\}$.
\end{proof}

\begin{theorem} \label{thm:commuting-lattes}
Let $f$ and $g$ be two commuting polynomial endomorphisms of $\C^2$ of degree $d_f, d_g \geq 2$ that extend holomorphically to $\pr^2$. Assume that $f^n \neq g^m$ for all $ n,m \geq 1$ and that $d_f^k = d_g^\ell$ for some $ k,\ell \geq 1$. 
Then, in suitable coordinates, $f$ and $g$ are homogeneous and the induced maps on the hyperplane at infinity are Lattès maps of $\pr^1$.
\end{theorem}

\begin{proof}
Notice that $f^k$ and $g^\ell$ are commuting endomorphisms of the same degree. By Proposition \ref{prop:pol-one-leaf} and from the fact that the Julia set of $f$ and $g$ are the same as those of $f^k$ and $g^\ell$ we see that there is a nonempty subset $\Omega$ of $\pr^2$  such that $J_2 \cap \Omega$ is a real analytic hypersurface. From Theorem \ref{thm:dinh-lattes} there is a coordinate system on $\C^2$ such that $f$ is homogeneous and the restriction of $f$ to the line at infinty $L_\infty$ is a Lattès map. Notice that $\mathcal E_f = \{0\} \cup L_{\infty}$, so $\mathcal E_g = \{0\} \cup L_{\infty}$ from Proposition \ref{prop:same-E}.

From the form of the Poincaré map given by Theorem \ref{thm:dinh-lattes} we see that the lines through the origin are given by the images of the parallel lines $\{z_1 = constant \}$ by $\varphi$. Since $\Lambda_g$ preserves this family we see that $g$ preserves the family of lines through the origin, so $g$ is also homogenous in these coordinates. As $g|_{L_\infty}$ commutes with $f|_{L_\infty}$ we get from Corollary \ref{cor:commuting-lattes} that $g|_{L_\infty}$ is a Lattès map.
\end{proof}


\subsection{Ruling out the non-polynomial case}
We consider now a commuting pair of holomorphic endomorphisms $f,g: \pr^2 \to \pr^2$ having same degree $d_f = d_g = d \geq 2$ and satisfying (\ref{eq:disjoint-iterates}).

As above, let $m$ be the maximal integer such that $\mu$ is laminated in some open subset of $\pr^2$. We will prove that if $m=3$ then both $f$ and $g$ are polynomial maps after some number of iterations. We thus fall into the case treated in the last section.

\begin{theorem} \label{thm:3-laminated-polynomial}
Let $f,g: \pr^2 \to \pr^2$ be commuting endomorphisms of degree $d$ such that $f^n \neq g^n$ for all $ n \geq 1$ and let $\mu$ be their common Green measure. Suppose that there is an open set $\Omega$ of $\pr^2$ such that $\mu$ is $3$-laminated in $\Omega$ and that $\mu$ is nowhere $4$-laminated (i.e.\ smooth). Then $f$ and $g$ are polynomial after some number of iterations.
\end{theorem}

Recall from Theorem \ref{thm:exceptional-set} that if the exceptional set $\mathcal E$ of a holomorphic endomorphism $f$ of $\pr^2$ is positive-dimensional then the one-dimensional part of $\mathcal E$ is a union of $1$, $2$ or $3$ lines. In this case $f^3$ possess  a totally invariant line, so it must be a polynomial endomorphism. Therefore, in order to prove Theorem \ref{thm:3-laminated-polynomial}, we need to show that under the above hypothesis, the common exceptional set $\mathcal E$ of $f$ and $g$ is positive-dimensional.

The proof will be divided in two parts. In the first part we will show that the leaves coming from the lamination of $\mu^*$ cannot be Levi-flat. This leaves us to the case where the leaves are strictly pseudoconvex at a generic point, which will be  considered in section  \ref{subsection:pseudonvex}.

\begin{remark} \label{rmk:dichotomy}
The leaves of the lamination of $\mu$ and $J$ are either all Levi-flat or all strictly pseudoconvex at a generic point. Indeed, suppose that there is one Levi-flat leaf $\mathcal L$ in $\Omega$ and take a point $a \in \mathcal L \setminus \mathcal E$. If $b$ any point in $J$, the equidistribution of pre-images (Theorem \ref{thm:equidistribution-preimages}) says that there is a point $b'$ in a small neighborhood of $b$ that is mapped to $a$ by some iterate $f^{n_0}$. Since $(f^{n_0})^* \mu = d^{2n_0} \mu$, the map  $f^{n_0}$ must preserve the leaves, so it maps a piece of the leaf $\mathcal L '$ through $b'$ to $\mathcal L$. Therefore $\mathcal L'$ is Levi-flat in a neighborhood of $b'$. Recall that $\mathcal L'$ is real analytic. Since $b'$ can be taken arbitrarily close to $b$ we have, by continuity, that the leaf through $b$ is Levi-flat.
\end{remark}

\subsubsection{Lamination by Levi-flat hypersurfaces} 

The aim of this section is to show that the Green measure cannot be laminated by Levi-flat hypersurfaces.

Let $f,g,\mu$ and $\Omega$ be as in Theorem \ref{thm:3-laminated-polynomial}. Pick a common repelling periodic point $a \in \Omega \cap J$ and consider the corresponding Poincaré map $\varphi: \C^2 \to \pr^2$.

\begin{lemma} \label{lemma:levi-flat-julia}
Suppose that $\mu$ is laminated by Levi-flat hypersurfaces. Then the leaves of the lamination of $\mu^* = \varphi^* \mu$ are given by $S_c = \{\im z_2 = c\}$ where $c$ varies on a closed set $K \subset \R$.
\end{lemma}

\begin{proof}
Let us first show that the central leaf is $S_0 = \{\im z_2 = 0\}$. We can choose coordinates as in Lemma \ref{lemma:circle-group} such that $H = T_0 S_0 = \{\im z_2 = 0\}$, $\Lambda_f (z_1,z_2) = (\lambda_1 z_1, \lambda_2 z_2)$ with $\lambda_2 \in \R$ and $\Lambda^t = (e^{it}z_1,z_2)$. We can show as in the proof of Lemma \ref{lemma:circle-group} that $S_0$ is invariant by $\Lambda_f$ and $\Lambda_g$. In particular, $S_0$ is invariant by $\Lambda =\bigl( \begin{smallmatrix} \lambda & 0 \\  0 & \lambda_2 \end{smallmatrix} \bigr)$ for some $\lambda \in \R.$

We can write $S_0$ as a graph of the form $\im z_2 = h(\re z_1, \im z_1, \re z_2)$ for some real analytic function $h$ vanishing to order $2$ at $0$. The invariance of $S_0$ by $\Lambda^t$ implies that $h \circ \Lambda^t = h$. Comparing the terms of the Taylor expansion of both sides we see that $h$ cannot contain any monomial of the form $z_1^n \bar z_1^m (\re z_2)^k$ with $n \neq m$. We can write then $h(\re z_1, \im z_1, \re z_2) = \psi(|z_1|^2, \re z_2)$ for some real analytic function $\psi$ that vanishes at the origin. The invariance of $S_0$ by $\Lambda$ implies that $\psi \circ \Lambda = \lambda_2 \psi$. Using this relation and comparing each term of the Taylor expansion of $\psi$ it is possible to show that $\psi$ has the form $\psi = \alpha |z_1|^{2\ell}$ for some $\alpha \in \R$, if there is an integer $\ell$ satisfying $\lambda^{2\ell} = \lambda_2$ or $\psi \equiv 0$ if no such $\ell$ exists. Since $S_0$ is Levi-flat  the first case is only possible if $\alpha = 0$, so $\psi \equiv 0$ also in this case. This shows that $S_0$ has de desired form. 

We can follow similar arguments as in the end of the proof of Lemma \ref{lemma:disjoint-spheres}. In the present case we get that  $S_c = \{\im z_2 = c\}$.
\end{proof}

\begin{proposition} \label{prop:J-pseudoconvex}
Let $f,g$, $\mu$ and $\Omega$ as in Theorem \ref{thm:3-laminated-polynomial}. Then the leaves of the lamination of $\mu$ in $\Omega$ are strictly pseudoconvex at a generic point.
\end{proposition}

\begin{proof}
Suppose by contradiction that there is a Levi flat leaf $\mathcal L \subset J \cap \Omega$ and take a repelling periodic point $a \in \mathcal L$. Consider the associated Poincaré map $\varphi: \C^2 \to \pr^2$. From Remark \ref{rmk:dichotomy} all the leaves are Levi-flat, so from Lemma \ref{lemma:levi-flat-julia} we have $J^* = \varphi^{-1}(J_2) = \cup_{c \in K} \{\im z_2 = c\}$.

{\bf Claim:} We have $J^* = \C^2$ and $J_2 = \pr^2$.

Since $\mathcal L$ is foliated by complex curves, Proposition \ref{prop:J1=J2} implies that $J_1 = J_2$. If there is a $c_0 \in \R \setminus K$ then $\varphi(\{z_2 = i c_0\})$ is an entire curve contained in $\pr^2 \setminus (\mathcal E \cup J_1)$, which is in turn contained in the Fatou set of $f$. This is absurd because the Fatou set is hyperbolic \cite{dinh-sibony:cime}. We must have then $K = \R$ and $J^* = \C^2$. Since $J_2$ is closed we have $J_2 = \overline{\varphi(\C^2)} = \pr^2$.
 
Considering the groupoid $\mathcal A$ of local biholomorphic maps acting on the fibers of $\varphi$ (as in the proof of Lemma \ref{lemma:S0totreal}) we see that its elements preserve the real foliation   $\{\im z_2 = cst.\}$, so they must also preserve the holomorphic foliation $\{z_2 = cst.\}$. If we let $A$ be complement of the image the of regular values of $\varphi$ we get then an $f$-invariant non-singular foliation $\mathcal F$ on $\pr^2 \setminus A$. This is impossible by Lemma \ref{lemma:no-hol-fol} (since $J_2 = \pr^2$ in our case). This gives the contradiction we were looking for.
\end{proof}

\subsubsection{Lamination by strictly pseudoconvex hypersurfaces} \label{subsection:pseudonvex}
In this section we suppose that the Green measure is $3$-laminated by leaves that are strictly pseudoconvex at a generic point.

Observe that  on the proof Lemma \ref{lemma:disjoint-spheres} have not used the fact that the maps are polynomial. We only needed the strict pseudoconvexity of the central leaf at $0$. By assumption we can choose a repelling periodic point $a \in \Omega \cap J$ at which the leaf that passes through it is strictly pseudoconvex. From this fact and Lemmas \ref{lemma:circle-group} and \ref{lemma:disjoint-spheres} we can assume, after a holomorphic change of coordinates and taking iterates,  that:
\begin{itemize}
\item The closed Lie group generated by $\Lambda_f \circ \Lambda_g^{-1}$ contains $\Lambda^t = \bigl( \begin{smallmatrix} e^{it} & 0 \\ 0 & 1 \end{smallmatrix} \bigr)$ as a $1$-parameter subgroup.
\item The maps $\Lambda_f$ and $\Lambda_g$ are linear and diagonal.
\item The eigenvalues of $\Lambda_f$ and $\Lambda_g$ satisfy $|\lambda_1|^2 = |\gamma_1|^2 = \lambda_2 = \gamma_2$.
\item The leaves of the lamination of $\mu^*$ are given by $S_c = \{\im z_2 = |z_1|^2 + c \}$, where the parameter $c$ varies on a closed set $K \subset \R$.
\end{itemize}

As in the previous section, we set $\lambda = |\lambda_1|^2$ and $\Lambda = \bigl( \begin{smallmatrix} \sqrt \lambda & 0 \\ 0 & \lambda \end{smallmatrix} \bigr) $. Since $\Lambda = \Lambda_f \circ \Lambda^\theta$  for some $\theta \in \R$ we have the invariance relations $\Lambda^* T^* = d\cdot T^*$, $\Lambda^* \mu^* = d^2 \cdot \mu^*$ and $\Lambda(J^*) = J^* = \Lambda^{-1}(J^*)$.

The following result is classical. It follows for instance from \cite{alexander:ball-polydisc} and the classification of automorphisms of the unit ball of $\C^2$, which can be found in \cite{rudin:unit-ball}.

\begin{proposition} \label{prop:tau}
Let $W$ be an open subset of $\C^2$ intersecting $J^*$ and $\tau:W \to \tau(W)$ a holomorphic map preserving the family $S_c = \{\im z_2 = |z_1|^2 + c \}$, $c \in K$.	

\begin{enumerate}
\item If $0 \in W$ and $ \tau(0) = 0$ then $\tau$ is linear of the form $\tau = \bigl( \begin{smallmatrix} \sqrt \delta e^{i \rho} & 0 \\ 0 & \delta \end{smallmatrix} \bigr) $ for some $\delta >0$ and some $\rho \in \R$.
\item If $0 \in W$ and $\tau(0) = (0,ic)$ then $\tau$ is affine of the form $\tau = t_{ic} \circ A $ where $A = \bigl( \begin{smallmatrix} \sqrt \delta e^{i \rho} & 0 \\ 0 & \delta \end{smallmatrix} \bigr)$ for some $\delta >0$ and some $\rho \in \R$ and $t_a: (z_1,z_2) \mapsto  (z_1,z_2 + a) $.
\item If $0 \in W$ then $\tau = \tau_u \circ t_{ic} \circ A$, where $A$ and $t_{ic}$ are as above and $\tau_u$ is the Heisenberg map:
\begin{equation}
\tau_u (z_1, z_2) = (z_1 + u_1,z_2 + 2i z_1 \overline{u}_1 + u_2),
\end{equation}
with $u = (u_1,u_2) \in S_0$.
\item In general, if $\tau$ takes a point of $S_c$ to a point of $S_{c'}$ we have $\tau = \tau_{u'} \circ t_{ic'} \circ A \circ t_{-ic} \circ \tau_u^{-1}$, where $A$, $t_{ic}$ $t_{ic'}$, $\tau_u$, $\tau_{u'}$ are as above with $u,u' \in S_0$.
\end{enumerate}

In particular the map $\tau:W \to \tau(W)$ extends to a globally defined affine map of $\C^2$.
\end{proposition}

Consider now the groupoid of local biholomorphisms acting on the fibers of $\varphi$. More precisely, set
\begin{equation*}
\mathcal A = \{\tau: \tau \text{ is a local biholomorphism of } \C^2 \text{ such that } \varphi\circ \tau = \varphi \}.
\end{equation*}

\begin{proposition}  \label{prop:description-A}
The set $\mathcal A$ is a discrete subgroup of the group $\text{\normalfont Aff}(\C^2)$ of affine transformations of $\C^2$. Furthermore $\mathcal A$ is contained in the group generated by the following types of maps:
\begin{itemize}
\item [i)] purely imaginary translations along the $z_2$-axis: $t_{ic} (z_1,z_2) = (z_1,z_2 + ic) $,
\item [ii)] rotations around the $z_2$-axis: $r_\rho(z_1,z_2) = (e^{i\rho} z_1,z_2)$ and
\item [iii)] Heisenberg maps: $\tau_u (z_1, z_2) = (z_1 + u_1,z_2 + 2i z_1 \overline{u}_1 + u_2)$ with $u = (u_1,u_2) \in S_0$.
\end{itemize}
\end{proposition}

\begin{proof}
Let $\tau \in \mathcal A$. Notice that $\tau$ must preserve $\mu^*$ so, by similar arguments as in the proof of Lemma \ref{lemma:circle-group},  it must preserve the leaf structure of $\mu^*$, i.e.\ , the family  $S_c = \{\im z_2 = |z_1|^2 + c \}$, $c \in K$. From Proposition \ref{prop:tau}, $\tau $ is a globally defined invertible affine map. This shows that $\mathcal A$ is subgroup of $\text{Aff}(\C^2)$. The discreteness of $\mathcal A$ follows from the fact that $\varphi$ has discrete fibers.

Set $S_c = \tau(S_0)$. Then Proposition \ref{prop:tau} says that $\tau = \tau_u \circ t_{ic} \circ A$ with $ A = \bigl( \begin{smallmatrix} \sqrt \delta e^{i \rho} & 0 \\ 0 & \delta \end{smallmatrix} \bigr)$. To conclude the proof we must show that $\delta = 1$.

Suppose that $\delta \neq 1$. Replacing $\tau$ by $\tau^{-1}$ if necessary we may assume that $\delta < 1$. In this case the map $\tau$ has a fixed point $z \in \C^2$ which is attracting. In particular, there is an orbit $\tau^{n}(w)$ converging to $z$. This is absurd because the fibers of $\varphi$ are discrete.
\end{proof}

We are now in position to give the proof of Theorem \ref{thm:3-laminated-polynomial}.

\begin{proof}[Proof of Theorem \ref{thm:3-laminated-polynomial}] From the remarks made after the statement of Theorem \ref{thm:3-laminated-polynomial} the conclusion will follow if we show that $\dim \mathcal E = 1$.

By Proposition \ref{prop:J-pseudoconvex}, we can pick a repelling periodic point at which the leaf passing through it is strictly pseudoconvex. From the considerations made at the beginning of this section we have that the leaves of the lamination of $\mu^*$ are given by $S_c = \{\im z_2 = |z_1|^2 + c \}$, where the parameter $c$ varies on a closed set $K \subset \R$. There are two cases to consider. 

\textbf{Case I:} The lifted Julia set $J^*$ consists of only one leaf, that is, $K = \{0\}$.

In this case the group $\mathcal A$ from Proposition \ref{prop:description-A} is generated only by rotations $r_\rho$ and Heisenberg maps $\tau_u$. Therefore the p.s.h.\ function $v(z_1,z_2) =  |z_1|^2 - \im z_2$ on $\C^2$ is invariant by $\mathcal A$, so it descends to a p.s.h.\ function $u$ on $\pr^2 \setminus \mathcal E$. If  $\dim \mathcal E = 0$ then $u$ extends to  a  p.s.h.\ function on $\pr^2$, which is absurd because $\pr^2$ is compact. This contradiction shows that $\dim \mathcal E = 1$.

\textbf{Case II:} The lifted Julia set $J^*$ contains more than one leaf, that is, $K$ contains at least $2$ points.

Consider the p.s.h.\ function $v(z_1,z_2) =  |z_1|^2 - \im z_2$ on $\C^2$. It satisfies $t_{ic}^* v = v + c$, $\tau_u^* v = v$ and $r_\rho^* v = v$. These relations show that the positive closed $(1,1)$-current $S^* = \ddc v$ is invariant under $t_{ic}$, $r_\rho$ and $\tau_u$. Proposition \ref{prop:description-A} implies then that $S^*$ is invariant by the whole group $\mathcal A$, so it descends to a current $S$ in $\pr^2 \setminus \mathcal E$. More precisely, there is a positive closed $(1,1)$-current $S$ on $\pr^2 \setminus \mathcal E$ such that $\varphi^* S = S^*$.

Suppose that $\dim \mathcal E = 0$. Then $S$ extends to a positive closed $(1,1)$-current on $\pr^2$, which we still denote by $S$.

\textbf{Claim.} The eigenvalues of $\Lambda_f$ and $\Lambda_g$ satisfy $|\lambda_1|^2 = |\gamma_1|^2 = \lambda_2 = \gamma_2 = d$. The current $S$ satisfies $f^*S = d \cdot S$.

Firstly we observe that $f^*S = \lambda \cdot S$ on $\pr^2 \setminus \mathcal E$. Indeed, we have the relation $\Lambda_f^* S^* = \lambda \cdot S^*$, which gives $$\varphi^* f^*S = \Lambda_f^* \varphi^*S = \Lambda_f^* S^* = \lambda \cdot S^* = \varphi^*(\lambda \cdot S),$$ so $f^*S = \lambda S$ on $\pr^2 \setminus \mathcal E$. Now, the $(1,1)$-current $f^*S - \lambda S$ is supported by $\mathcal E$, so it must be identically zero. Therefore $f^*S = \lambda S$ on $\pr^2$. On one hand we have $\|f^*S\| = \|\lambda S\| = \lambda \|S\|$. On the other hand, since $f^*$ acts by multiplication by $d$ on $H^{1,1}(\pr^2)$ we have $\|f^*S\| = d \|S\|$. Hence $\lambda = d$, proving the claim (recall that $\lambda = |\lambda_1|^2$).

\vskip4pt

After multiplying $S$ by a constant we may assume that $S$ is of mass $1$. Write $T = S - \ddc  u$, where $T$ is the Green current of $f$ and $u$ is a p.s.h.\ function modulo $T$ defined up to an additive constant. We may normalize it so that $\int  u \; d \mu = 0$. Since $T$ is of bounded local potential, the function $u$ is locally bounded from above. The invariance of $S$ and $T$ implies that $u \circ f^n = d^n \cdot u$, so we must have $u \leq 0$. The chosen normalization implies that $u= 0$ on $J_2$.

Consider the function $u^* = u \circ \varphi$. It vanishes on $J^*$  and satisfies  $T^* = S^* - \ddc  u^*$. Notice that $ S^* = \ddc(|z_1|^2 - \im z_2) =  \ddc |z_1|^2 = \frac{i}{\pi} dz_1 \wedge d\bar z_1$ up to a constant multiple. In particular $S^*$ restricts to zero on $L_0= \{z_1 = 0\}$. We have then that $T^*|_{L_0} = - \ddc u^*_0$, where $u_0^*(z_2) = u^*(0,z_2)$. The function $u_0$ is then a continuous negative superharmonic function on $L_0$ satisfying $u_0^* (d z_2) = d \cdot u_0^* (z_2)$ and vanishing on the lines $\{\im z_2 = c\},  c \in K$. Lemma \ref{lemma:subharmonic-upper-half-plane} below and the fact that $K$ contains more then one point imply that $u_0^* \equiv 0$. Therefore $u^*$ vanishes over $L_0$ and $u$ vanishes over $\varphi(L_0)$. Since $u \leq 0$ and is upper semi-continuous it vanishes on $\overline{\varphi(L_0)}$.

The restriction of the Poincaré map to $L_0$ gives an entire curve $\varphi|_{L_0}: L_0 \simeq \C \to \pr^2$. The Ahlfors construction (Theorem \ref{thm:ahlfors-construction}) produces  a positive closed $(1,1)$-current $R$ on $\pr^2$ that is supported in $\overline{\varphi(L_0)}$. Consider the current $R ^* = \varphi^*R$. It is a limit of currents $R_n$ supported in $\varphi^{-1}(\varphi(\Delta_n)) = \cup_{\tau \in \mathcal A} \, \tau(\Delta_n)$, where $\Delta_n$ is an increasing sequence of discs contained in $L_0$. By Proposition \ref{prop:description-A} this set is a discrete union of discs contained in vertical lines $\{z_1 = c_\tau\}$, $\tau \in \mathcal A$. In particular $i  dz_1 \wedge d\bar z_1 \wedge R_n= 0$ for every $n$, so $i  dz_1 \wedge d\bar z_1 \wedge R= 0$. Therefore $$\varphi^* (S \wedge R) = C^{st} \cdot i  dz_1 \wedge d\bar z_1 \wedge R^* = 0,$$
so the measure $S \wedge R$ is supported by $\mathcal E$.

On the other hand we have seen that $u$ vanishes on $\overline{\varphi (L_0)}$. In particular, $u$ vanishes on the support of $R$. This implies that $T \wedge R = S \wedge R$. The above discussion  shows that $T \wedge R $ is supported by $\mathcal E$. This gives a contradiction, because, since $T$ has  Hölder continuous local potentials, the measure $T \wedge R $ cannot charge points (see \cite{sibony:dynamique-Pk}). 
\end{proof}

The following result can be found in \cite{dinh-sibony:endo-permutables}, Lemme 5.3.
\begin{lemma} \label{lemma:subharmonic-upper-half-plane}
Let $v \geq 0$ be a continuous subharmonic function defined on the closed upper half-plane $\overline{\mathbb H}$. Suppose that $v$ vanishes on $\{\im z = 0\}$ and that there is a $d > 1$ such that $v(dz) = d\cdot v(z)$ for every $z \in \overline{\mathbb H}$. Then there is a constant $\beta \geq 0$ such that $v(z) = \beta \im z$.
\end{lemma}

\section{Proof of the Main Theorem}
We are able now to give a proof of Theorem \ref{thm:main-thm}.

\begin{proof}[Proof of Theorem \ref{thm:main-thm}] Let $f$ and $g$ be commuting holomorphic endomorphisms of $\pr^2$ such that $f^n \neq g^m$ for all $ n,m \geq 1$ and that $d_f^k = d_g^\ell$ for some $ k,\ell \geq 1$. Denote by $J_2$ the common  Julia set of $f$ and $g$ and by $\mu$ their Green measure. Notice that $f^k$ and $g^\ell$ are commuting endomorphisms of same degree. By Proposition \ref{prop:lamination} and from the fact that the Julia set and the Green measure of $f$ and $g$ are the same as those of $f^k$ and $g^\ell$ we see that $J_2$ and $\mu$ are laminated in some open subset of $\pr^2$. Furthermore, the maximal $m$ such that $J_2$ and $\mu$ are $m$-laminated in some open subset $\Omega$ of $\pr^2$ is $m=3$ or $4$.

If $m=3$, Theorem \ref{thm:3-laminated-polynomial} implies that $f^N$ and $g^M$ are polynomial for some $N,M \geq 1$, and Theorem \ref{thm:commuting-lattes} shows that they are lifts of Lattès maps of $\pr^1$. In particular, the common exceptional set of $f$ and $g$, which is the same as the exceptional set of $f^N$ and $g^M$, is a union of a point and a line. This shows that $f$ and $g$ are also polynomial, so by Theorem \ref{thm:commuting-lattes} again, $f$ and $g$ are  lifts of Lattès maps of $\pr^1$.

 If $m=4$, the Green measure is analytically equivalent to the Lebesgue measure in $\Omega$ so, by Theorem \ref{thm:m=4-lattes},  $f$ and $g$ are Lattès maps.
\end{proof}

\bibliography{refs-commuting-pairs}

\begin{thebibliography}{CLN00}

\bibitem[Ale74]{alexander:ball-polydisc}
H.~Alexander.
\newblock Holomorphic mappings from the ball and polydisc.
\newblock {\em Math. Ann.}, 209:249--256, 1974.

\bibitem[BD99]{briend-duval:liapounoff}
Jean-Yves Briend and Julien Duval.
\newblock Exposants de {L}iapounoff et distribution des points p\'eriodiques
  d'un endomorphisme de {$\bold C{\rm P}^k$}.
\newblock {\em Acta Math.}, 182(2):143--157, 1999.

\bibitem[BD01]{briend-duval:mesure}
Jean-Yves Briend and Julien Duval.
\newblock Deux caract\'erisations de la mesure d'\'equilibre d'un endomorphisme
  de {${\rm P}^k(\bold C)$}.
\newblock {\em Publ. Math. Inst. Hautes \'Etudes Sci.}, (93):145--159, 2001.

\bibitem[BD05]{berteloot-dupont:lattes}
F.~Berteloot and C.~Dupont.
\newblock Une caract\'erisation des endomorphismes de {L}att\`es par leur
  mesure de {G}reen.
\newblock {\em Comment. Math. Helv.}, 80(2):433--454, 2005.

\bibitem[BL01]{berteloot-loeb:lattes}
Fran{\c{c}}ois Berteloot and Jean-Jacques Loeb.
\newblock Une caract\'erisation g\'eom\'etrique des exemples de {L}att\`es de
  {${\Bbb P}^k$}.
\newblock {\em Bull. Soc. Math. France}, 129(2):175--188, 2001.

\bibitem[CLN00]{cerveau-lins-neto:hyp-exc}
D.~Cerveau and A.~Lins~Neto.
\newblock Hypersurfaces exceptionnelles des endomorphismes de {${\bf C}{\rm
  P}(n)$}.
\newblock {\em Bol. Soc. Brasil. Mat. (N.S.)}, 31(2):155--161, 2000.

\bibitem[DG12]{dujardin-guedj:maximal}
Romain Dujardin and Vincent Guedj.
\newblock Geometric properties of maximal psh functions.
\newblock In {\em Complex {M}onge-{A}mp\`ere equations and geodesics in the
  space of {K}\"ahler metrics}, volume 2038 of {\em Lecture Notes in Math.},
  pages 33--52. Springer, Heidelberg, 2012.

\bibitem[Din01]{dinh:lattes}
Tien-Cuong Dinh.
\newblock Sur les applications de {L}att\`es de {${\Bbb P}^k$}.
\newblock {\em J. Math. Pures Appl. (9)}, 80(6):577--592, 2001.

\bibitem[DS02]{dinh-sibony:endo-permutables}
Tien-Cuong Dinh and Nessim Sibony.
\newblock Sur les endomorphismes holomorphes permutables de $\pr^k$.
\newblock {\em Mathematische Annalen}, 324:33--70, 2002.

\bibitem[DS08]{dinh-sibony:equidistribution-ENS}
Tien-Cuong Dinh and Nessim Sibony.
\newblock Equidistribution towards the {G}reen current for holomorphic maps.
\newblock {\em Ann. Sci. \'Ec. Norm. Sup\'er. (4)}, 41(2):307--336, 2008.

\bibitem[DS10]{dinh-sibony:cime}
Tien-Cuong Dinh and Nessim Sibony.
\newblock Dynamics in several complex variables: endomorphisms of projective
  spaces and polynomial-like mappings.
\newblock In {\em Holomorphic dynamical systems}, volume 1998 of {\em Lecture
  Notes in Math.}, pages 165--294. Springer, Berlin, 2010.

\bibitem[Dup03]{dupont:lattes}
Christophe Dupont.
\newblock Exemples de {L}att\`es et domaines faiblement sph\'eriques de {$\Bbb
  C^n$}.
\newblock {\em Manuscripta Math.}, 111(3):357--378, 2003.

\bibitem[Ere89]{eremenko:functional-equations}
A.~{\`E}. Eremenko.
\newblock Some functional equations connected with the iteration of rational
  functions.
\newblock {\em Algebra i Analiz}, 1(4):102--116, 1989.

\bibitem[Fat24]{fatou}
P.~Fatou.
\newblock Sur l'itération analytique et les substitutions permutables.
\newblock {\em Journal de Mathématiques Pures et Appliquées}, pages 1--50,
  1924.

\bibitem[FS94]{fornaess-sibony:dynamics1}
John~Erik Forn{\ae}ss and Nessim Sibony.
\newblock Complex dynamics in higher dimension. {I}.
\newblock {\em Ast\'erisque}, (222):5, 201--231, 1994.
\newblock Complex analytic methods in dynamical systems (Rio de Janeiro, 1992).

\bibitem[FS01]{fornaess-sibony:examples}
John~Erik Forn{\ae}ss and Nessim Sibony.
\newblock Dynamics of {${\bf P}^2$} (examples).
\newblock In {\em Laminations and foliations in dynamics, geometry and topology
  ({S}tony {B}rook, {NY}, 1998)}, volume 269 of {\em Contemp. Math.}, pages
  47--85. Amer. Math. Soc., Providence, RI, 2001.

\bibitem[Gue03]{guedj:equidistribution}
Vincent Guedj.
\newblock Equidistribution towards the {G}reen current.
\newblock {\em Bull. Soc. Math. France}, 131(3):359--372, 2003.

\bibitem[Jul22]{julia}
Gaston Julia.
\newblock Mémoire sur la permutabilité des fractions rationnelles.
\newblock {\em Annales scientifiques de l'École Normale Supérieure},
  39:131--215, 1922.

\bibitem[Mil06]{milnor:lattes}
John Milnor.
\newblock On {L}att\`es maps.
\newblock In {\em Dynamics on the {R}iemann sphere}, pages 9--43. Eur. Math.
  Soc., Z\"urich, 2006.

\bibitem[Rit23]{ritt}
J.~F. Ritt.
\newblock Permutable rational functions.
\newblock {\em Trans. Amer. Math. Soc.}, 25(3):399--448, 1923.

\bibitem[Ron10]{rong:lattes}
Feng Rong.
\newblock Latt\`es maps on {$\bold P^2$}.
\newblock {\em J. Math. Pures Appl. (9)}, 93(6):636--650, 2010.

\bibitem[Rud08]{rudin:unit-ball}
Walter Rudin.
\newblock {\em Function theory in the unit ball of {$\Bbb C^n$}}.
\newblock Classics in Mathematics. Springer-Verlag, Berlin, 2008.
\newblock Reprint of the 1980 edition.

\bibitem[Sib99]{sibony:dynamique-Pk}
Nessim Sibony.
\newblock Dynamique des applications rationnelles de {$\bold P^k$}.
\newblock In {\em Dynamique et g\'eom\'etrie complexes ({L}yon, 1997)},
  volume~8 of {\em Panor. Synth\`eses}, pages ix--x, xi--xii, 97--185. Soc.
  Math. France, Paris, 1999.

\bibitem[Taf11]{taflin:equidistribution}
Johan Taflin.
\newblock Equidistribution speed towards the {G}reen current for endomorphisms
  of {$\Bbb P^k$}.
\newblock {\em Adv. Math.}, 227(5):2059--2081, 2011.

\bibitem[TY82]{tokunaga-yoshida}
Syoshi Tokunaga and Masaaki Yoshida.
\newblock Complex crystallographic groups. {I}.
\newblock {\em J. Math. Soc. Japan}, 34(4):581--593, 1982.

\bibitem[Ves87]{veselov:integrable}
A.~P. Veselov.
\newblock Integrable mappings and {L}ie algebras.
\newblock {\em Dokl. Akad. Nauk SSSR}, 292(6):1289--1291, 1987.

\end{thebibliography}
\bibliographystyle{alpha}
\end{document}